\documentclass[10pt, reqno]{amsart}
\usepackage{graphicx} 

\usepackage{amssymb}
\usepackage{amsthm}
\usepackage{quiver}
\usepackage{tikz}
\usetikzlibrary{decorations.markings}
\usetikzlibrary{arrows, positioning, patterns, patterns.meta}
\usepackage{xcolor}
\usepackage{float}
\usepackage{mathdots}
\usepackage{mathtools}

\usepackage[top=1.2in,bottom=1.2in,left=1.2in,right=1.2in]{geometry}

\definecolor{mygrey}{gray}{0.84}

\newcommand{\inv}{^{-1}}

\newcommand{\C}{\mathbb{C}}
\newcommand{\Z}{\mathbb{Z}}
\newcommand{\N}{\mathbb{N}}
\newcommand{\R}{\mathbb{R}}
\newcommand{\Q}{\mathbb{Q}}
\newcommand{\Ha}{\mathbb{H}}

\DeclareMathOperator{\BS}{BS}

\DeclareMathOperator{\Isom}{Isom}

\DeclareMathOperator{\Aut}{Aut}
\DeclareMathOperator{\Axis}{Axis}
\DeclareMathOperator{\Fix}{Fix}
\DeclareMathOperator{\fix}{Fix}

\DeclareMathOperator{\orb}{Orb}

\DeclareMathOperator{\Stab}{Stab}
\DeclareMathOperator{\Aff}{Aff}

\DeclareMathOperator{\SL}{SL}
\DeclareMathOperator{\down}{down}
\DeclareMathOperator{\Out}{Out}

\DeclareMathOperator{\GL}{GL}

\DeclareMathOperator{\up}{up}
\DeclareMathOperator{\td}{td}

\newtheorem*{lemma*}{Lemma} 
\newtheorem{theorem}{Theorem}
\newtheorem{lemma}{Lemma}[section]

\newtheorem{proposition}[lemma]{Proposition}
\newtheorem{corollary}[lemma]{Corollary}
\newtheorem*{corollary*}{Corollary}

\title{Solvable Baumslag-Solitar Lattices}
\author{Noah Caplinger }
\date{}

\begin{document}

\begin{abstract}
    The solvable Baumslag Solitar groups $\BS(1,n)$ each admit a canonical model space, $X_n$. We give a complete classification of lattices in $G_n = \Isom^+(X_n)$ and find that such lattices fail to be strongly rigid---there are automorphisms of lattices $\Gamma \subset G_n$ which do not extend to $G_n$---but do satisfy a weaker form of rigidity: for all isomorphic lattices $\Gamma_1,\Gamma_2\subset G_n$, there is an automorphism $\rho \in \Aut(G_n)$ so that $\rho(\Gamma_1) = \Gamma_2$. 
\end{abstract}

\maketitle

\vspace{-1cm}
\section{Introduction}

This paper is a case study in the rigidity of lattices in a particular non-linear locally compact group. For $n \geq 2$, let $C_n$ be the mapping torus of $S^1 \overset{n}{\to}S^1$. Then the fundamental group $\pi_1(C_n)$ is isomorphic to the \emph{solvable Baumslag-Solitar group} $$\BS(1,n) = \langle a,b \mid bab\inv = a^n\rangle.$$ There is a natural piecewise-Riemannian structure on $C_n$ given by identifying the sides of a \textit{horobrick} $[0,n]\times [1,n] \subset \Ha^2$ (in the upper half-plane model) in the manner described by Figure \ref{fig:IntroPic}. Let $X_n$ denote the universal cover of $C_n$ with the lifted piecewise-Riemannian metric. Topologically, $X_n$ is the product $\R \times T_{1,n}$, where $T_{1,n}$ denotes the bi-regular oriented tree with $1$ incoming and $n$ outgoing edges at each vertex. For each oriented geodesic $\ell \subset T_{1,n}$, the preimage $\pi\inv(\ell)$ under the projection $\pi:\R \times T_{1,n} \to T_{1,n}$ is isometric to $\Ha^2$ (see Section \ref{sec:XnPrelim} for a detailed construction of $X_n$). Dymarz \cite{Dymarz} showed that every locally compact group $H$ containing $\BS(1,n)$ as a uniform lattice maps to $\Isom(X_n)$ with compact kernel. It is then natural to ask which lattices appear in $\Isom(X_n)$. In this paper, we study lattices in the group of \emph{orientation-preserving} isometries $$G_n\coloneq\Isom^+(X_n).$$

\begin{figure}
    \centering

\begin{tikzpicture}[scale = .7]

\fill[fill={mygrey}] (0,1) -- (0,2) -- (2,2) -- (2,1) -- cycle;
    \draw  (-3,0) -- (4,0);
    \draw (0,0) -- (0,3.5);
    \node at (0,1)[circle,fill,inner sep=1pt]{};
    \node at (0,2)[circle,fill,inner sep=1pt]{};
    \node at (1,1)[circle,fill,inner sep=1pt]{};
    \node at (2,1)[circle,fill,inner sep=1pt]{};
    \node at (2,2)[circle,fill,inner sep=1pt]{};
    \node at (-2,3)[]{$\mathbb{H}^2$};

    \tikzset{->-/.style={decoration={
  markings,
  mark=at position #1 with {\arrow[scale=1]{>}}},postaction={decorate}}}

 \draw[->-=.6] (0,1) to  (1,1);
 \draw[->-=.6] (1,1) to  (2,1);
 \draw[->-=.54] (0,2) to  (2,2);
 \draw[->-=.6] (0,1) to  (0,2);
 \draw[->-=.5] (0,1) to  (0,2);
\draw[->-=.6] (2,1) to  (2,2);
 \draw[->-=.5] (2,1) to  (2,2);
 
\end{tikzpicture}
\hspace{1cm}
\begin{tikzpicture}[scale=.6]

\begin{scope}
    \clip (-2,.28) rectangle (2,-.85);
     \draw (0,0) ellipse (2cm and .5cm);

\end{scope}
     \draw[dashed] (0,0) ellipse (2cm and .5cm);
    \draw (0,2) ellipse (1cm and .25cm);
     \draw[->] (2.25,.5)   to[bend right] node[midway,above right] {$\times 2$} (1.5,1.75);

     \draw (-1,2) to [bend left] (-2,0);
     \draw (1,2) to [bend right] (2,0);
     

\end{tikzpicture}
\hspace{.5cm}
\begin{tikzpicture}[scale=.3]
    \draw (-2.588,-4.326) -- (-2.588,-2.982);
    \draw (-2.588,-4.326) -- (2.494,-3.623);
    \draw (2.494,-3.623) -- (2.485,-2.304);
    \draw (-2.588,-2.982) -- (-3.684,-1.214);
    \draw (-2.588,-2.982) -- (-1.95,-1.378);
    \draw (-1.95,-1.378) -- (3.127,-.683);
    \draw (3.127,-.683) -- (2.485,-2.304);
    \draw (-2.588,-2.982) -- (2.485,-2.304);
    \draw (-1.95,-1.378) -- (-1.747,0.484);
    \draw (-1.95,-1.378) -- (-0.19,0.28);
    \draw (-3.684,-1.214) -- (-1.91,-0.979);
    \draw (-3.684,-1.214) -- (-3.754,1.17);
    \draw (-3.684,-1.214) -- (-5.3,1.35);
    \draw (-1.747,0.484) -- (-1.176,2.29);
    \draw (-1.747,0.484) -- (-2.185,2.61);
    \draw (-1.747,0.484) -- (-0.163,0.692);
    \draw (-0.026,2.054) -- (-0.19,0.28);
    \draw (-0.19,0.28) -- (1.691,1.761);
    \draw (-0.19,0.28) -- (4.886,0.966);
    \draw (4.886,0.966) -- (3.127,-.683);
    \draw (-5.3,1.35) -- (-5.934,3.332);
    \draw (-5.3,1.35) -- (-6.946,3.664);
    \draw (-5.3,1.35) -- (-3.955,1.537);
    \draw (-3.754,1.17) -- (-4.78,3.094);
    \draw (-3.754,1.17) -- (-3.43,2.864);
    \draw (-1.88,4.36) -- (-6.946,3.664);
    \draw (-1.58,3.93) -- (-1.88,4.36);
    \draw (-0.86,4.02) -- (-5.934,3.332);
    \draw (0.28,3.796) -- (-4.78,3.094);
    \draw (-1.94,1.41) -- (-3.754,1.17);
    \draw (0.28,3.796) -- (-4.78,3.094);
    \draw (0.28,3.796) -- (0.488,3.41);
    \draw (-3.43,2.864) -- (1.65,3.565);
    \draw (-2.185,2.61) -- (2.88,3.313);
    \draw (-1.176,2.29) -- (3.894,2.974);
    \draw (-0.026,2.054) -- (5.05,2.75);
    \draw (1.691,1.761) -- (6.75,2.45);
    \draw (6.75,2.45) -- (4.886,0.966);
    \draw (-0.86,4.02) -- (-0.742,3.666);
    \draw (1.65,3.565) -- (1.577,3.139);
    \draw (2.88,3.313) -- (2.978,2.857);
    \draw (3.894,2.974) -- (3.774,2.578);
    \draw (5.05,2.75) -- (5,2.22);

    \node at (1.629,-.888)[circle,fill,inner sep=.75pt]{};
    \node at (-.757,-1.215)[circle,fill,inner sep=.75pt]{};
    \node at (.988,-2.504)[circle,fill,inner sep=.75pt]{};
    \node at (-1.396,-2.823)[circle,fill,inner sep=.75pt]{};

    \draw [dashed] (-1.396,-2.823) -- (-.757,-1.215);
    \draw [dashed] (.988,-2.504) -- (1.629,-.888);

\end{tikzpicture}

    \caption{ \footnotesize {(left) A horobrick with identifications defining $C_2$, (center) an approximately to-scale drawing of $C_2$, (right, stenciled from \cite{FarbMosherI}) a piece of $X_2$, the universal cover of $C_2$. Each vertical sheet is isometric to $\Ha^2$. The dashed rectangle is a fundamental domain and isometric to a horobrick. }}
    \label{fig:IntroPic}
\end{figure}
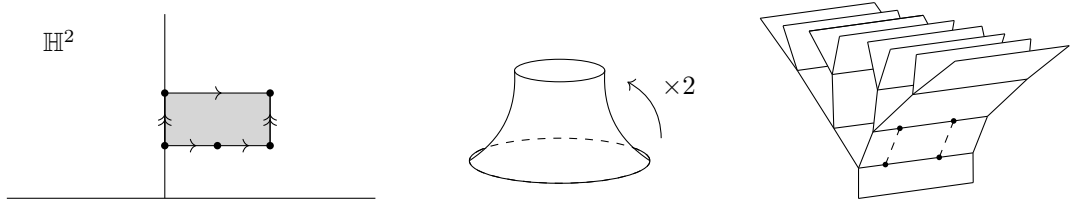

Of course, $\pi_1(C_n) \cong \BS(1,n)$ acts on $X_n$ by deck transformations, giving a `standard' lattice $\BS(1,n) \subset G_n$. One can also consider the finite index subgroups $\BS(1,n^l) = \langle a, b^l\rangle \subset \BS(1,n)$. Our first theorem says that these are the only abstract isomorphism classes of lattices in $G_n$.

\begin{theorem}[\textbf{Structure theorem}]
\label{thm:LatticesAreUniformAndBS}
Let $\Gamma \subset  G_n$ be a lattice. Then
\begin{enumerate}
    \item $\Gamma$ is uniform (that is, $G/\Gamma$ is compact)
    \item $\Gamma$ is abstractly isomorphic to $\BS(1,n^l)$ for some $l\geq 1$.    
\end{enumerate}
\end{theorem}

Lattices in the full isometry group $\Isom(X_n)$ are more subtle---by Theorem \ref{thm:LatticesAreUniformAndBS}, they are $\BS(1,n^l)$ extensions of $\Z/2$. We classify the isomorphism types of these lattices in Corollary \ref{cor:FullIsom}.

\subsection*{Formula for covolume.} The proof of Theorem \ref{thm:LatticesAreUniformAndBS} makes use of the following formula for the covolume of a discrete subgroup $\Gamma \subset G_n$ in terms of its induced actions on $T_{1,n}$ and $\Ha^2$. For a vertex $v \in T_{1,n}$, let $\Gamma_v = \Stab_{\Gamma}(v)$, let $h(v)$ denote its \textit{height} and let $a_v$ be the minimal positive \emph{translation distance} of elements of $\Gamma_v$ (see Section \ref{sec:XnPrelim}). If $\Gamma$ is discrete, the subgroup $\Gamma_{v,0} \subset \Gamma_v$ of elements with translation distance $0$ is finite. Let $\mu$ denote the Haar measure of $G_n$, normalized so that the standard lattice $\BS(1,n) \subset G_n$ has covolume $1$. We prove in Lemma \ref{lem:HaarFormula} that $$\mu_{ G_n/\Gamma}(G_n/\Gamma) = \sum_{[v]\in T/\Gamma} \frac{a_v\cdot n^{-h(v)}}{|\Gamma_{v,0}|}.$$ This formula is analogous to the formula for the covolume of a tree lattice, see \cite[Section 1.5]{TreeLattices}.

\subsection*{Classification of lattices.}

Theorem \ref{thm:LatticesAreUniformAndBS} leads us to investigate lattice embeddings $\BS(1,n^l)\to G_n$. The next two theorems give a complete classification of these lattice embeddings up to conjugacy (Theorem \ref{thm:Classification}) and post-composition by automorphisms (Theorem \ref{thm:Automorphism}, Corollary \ref{cor:Corollary}). We first introduce the lattice embeddings appearing in Theorem \ref{thm:Classification}.

One can modify the construction of $C_n$ by using a horobrick $[0,s\cdot n] \times [1,n]$ of different length. Any two choices of length give isometric universal covers, but different deck transformations. Specifically, $a \in \BS(1,n)$ has a different \textit{translation distance} (see Section \ref{sec:XnPrelim}) but the same action on the tree. The above-described action of $a \in \BS(1,n)$ with translation distance $s \in \R \setminus \{0\}$ (corresponding to the horobrick $[0,s\cdot n] \times [1,n]$) will be denoted $a_s \in G_n$. The action of $b \in \BS(1,n)$ does not depend on the length of the horobrick. This gives a one parameter family of lattices $\langle a_s, b\rangle \subset G_n$.

One can also modify the previous lattices embeddings $\langle a_s, b\rangle  \subset G_n$ by pre-composing with the injective endomorphism $\theta_m:\BS(1,n) \to \BS(1,n)$ given by $a \to a^m$ and $b \to b$. These endomorphisms have finite index image, restrict to endomorphisms of $\BS(1,n^l)$, and are automorphisms if and only if\footnote{Equivalently, if there is some $k,j \in \N$ so that $m = n^k/j$. In this case, $a = b^{-k}a^{mj}b^k$ lies in the image of $\theta_m$.} every prime factor of $m$ divides $n$. For each $l,m \in \N$ and $s \in \R \setminus \{0\}$, there is a lattice embedding $\phi_{s,m}: \BS(1,n^l) \to G_n$ given by $$\phi_{s,m}(a) = (a_s)^m \qquad \text{and} \qquad \phi_{s,m}(b^l) = b^l.$$

Let $C_b\inv \in \Aut(G_n)$ denote the inner automorphism corresponding to $b\inv \in G_n$. Our classification must account for the fact that if $n \mid m$, then $C_b\inv\circ \phi_{s,m} = \phi_{s,m/n}$. In fact, we will only need to deal with the case where $\theta_m$ is an isomorphism. Let $(\ast)$ denote these two conditions on $m \in \N$: $$(\ast) \qquad \text{Every prime factor of } m \text{ divides } n, \text{ and } n \nmid m.$$ We will show in the proof of Theorem \ref{thm:Classification} that if $(s,m)\neq (s',m')$ with $m$ and $m'$ both satisfying $(\ast)$, then the corresponding embeddings $\phi_{s,m}$ and $\phi_{s',m'}$ not conjugate. In fact, the pair $(s,m)$ is a complete conjugacy invariant of lattice embeddings.

\begin{theorem}[\textbf{Classification of lattice embeddings}]
\label{thm:Classification}
Let $\psi: \BS(1,n^l) \to G_n$ be a lattice embedding. Then there is an $s \in \R\setminus \{0\}$ and an $m \in \N$ satisfying $(\ast)$ so that $\psi$ is conjugate to $\phi_{s,m}$. The pair $(s,m)$ is a (complete) conjugacy invariant. 

\end{theorem}

In particular, when $m$ does not satisfy $(\ast)$, the map $\phi_{s,m}$ is conjugate to some other $\phi_{s',m'}$ with $m'$ satisfying $(\ast)$, and possibly $s \neq s'$. For instance, if $n = 2$, then $\phi_{1,3}$ is conjugate to $\phi_{3,1}$. 

When $n$ is prime, Theorem \ref{thm:Classification} is related to---but does not follow from---rigidity results for $S$-arithmetic lattices. We discuss some of these results and their relationship to Theorem \ref{thm:Classification} in Section \ref{sec:PrevWork}. 

\subsection*{Rigidity of Lattices.} When $\Gamma\subset G$ is a lattice in a semisimple Lie group, \emph{strong (Mostow) rigidity} gives sufficient conditions on $G$ for any automorphism $\phi: \Gamma \to \Gamma$ to extend to an automorphism $\rho:G\to G$ in the sense that the following diagram commutes. 
\[\begin{tikzcd}
	\Gamma & \Gamma \\
	G & G
	\arrow["\phi", from=1-1, to=1-2]
	\arrow[hook, from=1-1, to=2-1]
	\arrow[hook, from=1-2, to=2-2]
	\arrow["\rho", from=2-1, to=2-2]
\end{tikzcd}\]
In most cases $\Out(G) = 1$, so the automorphism $\rho$ may be taken to be inner, that is, conjugation by an element of $G$. Theorem \ref{thm:Classification} shows that in our case the automorphisms $\theta_m \in \Aut(\Gamma)$ do not extend to inner automorphisms of $G_n$. In the following theorem, we compute the full (topological) automorphism group of $G_n$, and find that while strong rigidity still fails, $G_n$ does satisfy a weaker form of \emph{subgroup rigidity}: for any two isomorphic lattices $\Gamma_1,\Gamma_2 \subset G_n$, there is an automorphism $\rho \in \Aut(G_n)$ so that $\rho(\Gamma_1) = \Gamma_2$. This phenomenon (subgroup rigidity, but not strong rigidity) does not appear in real or $p$-adic Lie groups. The author is not aware of any other instance of this phenomenon.

For any $r \in \R^\ast$, there is an automorphism $f_r$ that stretches translation distances by a factor of $r$, but preserves the action on the tree (see Section \ref{sec:AutGn} for a precise construction). These automorphisms satisfy $f_r\circ \phi_{s,m} = \phi_{rs,m}$, and represent the entire outer automorphism group. 

\begin{theorem}
\label{thm:Automorphism}
    Let $n \geq 2$. As topological groups, $\Aut(G_n) \cong \R^\ast \times \Aut(T_{1,n})$ and $\Out(G_n) \cong \R^\ast$.  
\end{theorem}

\noindent Combining Theorems \ref{thm:LatticesAreUniformAndBS}, \ref{thm:Classification} and \ref{thm:Automorphism} gives a complete picture of the lattices in $G_n$. 

\begin{corollary}[\textbf{Summary of results}]
\label{cor:Corollary}
Let $\psi:\Gamma \to G_n$ be a lattice embedding.
\begin{enumerate}
    \item (Theorem \ref{thm:LatticesAreUniformAndBS}) There is an $l \in \N$ so that $\Gamma \cong \BS(1,n^l)$.
    \item (Theorem \ref{thm:Classification}) $\psi$ is conjugate to a unique $\phi_{s,m}$ for $s \in \R \setminus \{0\}$ and $m$ satisfying $(\ast)$.
    \item There is an automorphism $\rho \in \Aut(G_n)$ so that $\rho \circ \psi = \phi_{1,m}$ for the same $m$ as above. 
    \item The automorphism $\theta_m$ of $\BS(1,n)$ does not extend to an automorphism of $G_n$: if $m \neq m'$ both satisfy $(\ast)$, then there does not exist $\rho \in \Aut(G_n)$ so that $\rho \circ \phi_{1,m} = \phi_{1,m'}$.
    \item If $\Gamma_1,\Gamma_2\subset G_n$ are isomorphic lattices, there exists a $\rho \in \Aut(G_n)$ so that $\rho(\Gamma_1) = \Gamma_2$.
\end{enumerate}
\end{corollary}

We give a full discussion of which automorphisms of $\BS(1,n)$ extend to $G_n$ in Section \ref{sec:AutGn}. We also use the above results to classify the isomorphism types of lattices in the full isometry group $\Isom(X_n)$.

\begin{corollary}
\label{cor:FullIsom}
    Let $\Gamma \subset \Isom(X_n)$ be a lattice. Then one of the following holds.
    \begin{enumerate}
        \item $\Gamma \subset G_n$, and therefore $\Gamma \cong \BS(1,n^l)$ for some $l \in \N$.
        \item There is an even $l \in \N$ so that $$\Gamma \cong \langle a, b,c \mid bab\inv = a^{n^l}, \; cac\inv = a^{-n^{l/2}}, \; c^2 = b  \rangle \cong \BS(1,-n^{l/2}).$$
        \item There is an $l \in \N$ and $y \in \Z$ so that $$\Gamma \cong \langle a,b,c \mid bab\inv = a^{n^l}, \; cac\inv = a\inv, \; cbc\inv = a^y b, \; c^2 = 1\rangle.$$
    \end{enumerate}
    In the final two cases, the lattice $\Gamma^+ \coloneq \Gamma \cap G_n$ is generated by $a$ and $b$ in the above presentations.
\end{corollary}

These groups correspond to the subgroups of $\Aff^{\pm}(\R)$ generated by $a(x) = x+1$, $b(x) = n^l\cdot x$, and (in case $2$) $c(x) = -n^{l/2}\cdot x$ or (in case $3$) $c(x) = -x + m$ for $m \in \Z$ (in which case $y = m(1-n)$).

\subsection*{Overview.} In Section \ref{sec:PrevWork}, we summarize some previous work on Baumslag-Solitar lattices. Section \ref{sec:XnPrelim} gives a precise construction of $X_n$ and establishes some notation and terminology. Section \ref{sec:LatticesAreUniform} contains the proof of Theorem \ref{thm:LatticesAreUniformAndBS}. We highlight in particular Lemma \ref{lem:EventuallyTransForever}, which gives a strong constraint on how elliptic elements of a lattice act on $T_{1,n}$. This is the first and most important rigidity lemma in the paper. Theorem \ref{thm:Classification} is proven in Section \ref{sec:ClassificationOfEmbeddings} by applying Lemma \ref{lem:EventuallyTransForever} to the image of $a \in \BS(1,n^l)$ under a lattice embedding and using the Baumslag-Solitar relation to gradually nail down its action on $T_{1,n}$.  We then compute $\Aut(G_n)$ in Section \ref{sec:AutGn}, and show that $G_n$ is not linear in Section \ref{sec:NotLinear}.

\subsection*{Acknowledgments.} I would like to thank Benson Farb for suggesting this problem, for his constant encouragement throughout this project and his extensive comments on early drafts. I would also like to thank Abhijit Mudigonda and Daniel Studenmund for their helpful conversations about $S$-arithmetic lattices, and Dan Margalit, Ethan Dlugie and Max Forester for their helpful comments and corrections.

\section{Previous work}
\label{sec:PrevWork}

Below we highlight some previous work on Baumslag-Solitar groups that relates to this paper.

\subsection*{Quasi-isometric rigidity of $\BS(1,n)$.}

The space $X_n$ first appeared in \cite{FarbMosherI} and \cite{FarbMosherII}, where it was used to prove the \textit{quasi-isometric rigidity} of $\BS(1,n)$---any group $\Lambda$ quasi-isometric to $\BS(1,n)$ is a finite extension $$1 \to F \to \Lambda \to \Omega \to 1$$ for $F$ finite, and $\Omega$ abstractly commensurable to $\BS(1,n)$. We use this result to derive part 2 of Theorem \ref{thm:LatticesAreUniformAndBS} from part 1. 

Farb-Mosher prove their theorem by developing a boundary theory of $X_n$ analogous to the boundary theory of hyperbolic space used in the proof of Mostow rigidity. The \emph{upper boundary} of $X_n$, denoted $\partial^uX_n$, is the space of sections of $X_n \to \Ha^2$, or equivalently the space of ``positive ends" of $T_{1,n}$. It is homeomorphic to the $n$-adic rational numbers, $\Q_n$. While we do not work explicitly with the upper boundary in this paper (instead dealing directly with the action on $T_{1,n}$), various steps of the proof of Theorem \ref{thm:Classification} can be phrased in terms of $\partial^uX_n$. For example: Lemma \ref{lem:EventuallyTransForever} specifies the local action of a power of $\phi(a)$ on $\partial^uX_n$ for some lattice embeddings $\phi: \BS(1,n)\to G_n$.

\begin{lemma*}[Lemma \ref{lem:EventuallyTransForever} rephrased]
    Let $\psi:\BS(1,n) \to G_n$ be a lattice embedding. Then there is some cantor set $K \subset \partial^u X_n \cong \Q_n$ and an identification of $K$ with $\Z_n$ so that some power of $\psi(a)$ acts by $x \to x+1$ on $K$. 
\end{lemma*}

Similarly, Lemma \ref{lem:HyperbolicConjugation} gives a copy of $\Z_n \subset \partial^uX_n$ on which $\psi(b)$ acts by $x \to nx$. The proof proceeds by gradually upgrading this data, eventually concluding that the action of $\BS(1,n)$ on $\partial^uX_n \cong \Q_n$ is given by $x \overset{a}{\to} x+m$ and $x \overset{b}{\to} nx$ for $m$ satisfying $(\ast)$.

\subsection*{$\mathbf{\BS(1,p)}$ as an $S$-arithmetic group and associated rigidity theorems.}

When $n = p$ is prime, 
\begin{equation}\label{Eqn:EmbeddingOfBS1p}\BS(1,p^2)\cong \left\{\begin{pmatrix}
    p^x & p^y \cdot z \\ 0 & p^{-x}
\end{pmatrix} \mid x,y,z \in \Z \right\} \underset{\text{index} \; 2}{\subset} \left \{  \begin{pmatrix}
    a & b \\ 0 & c
\end{pmatrix} \mid a,b,c \in \Z[1/p], \;  a\inv = c\right\} \end{equation} is an $S$-arithmetic group, so each $\BS(1,p^l)$ (commensurable to $\BS(1,p^2)$) is also $S$-arithmetic. We will now mention two rigidity theorems of $S$-arithmetic lattices and how they relate to our main theorems. 

\subsection*{Superrigidity of solvable $S$-arithmetic lattices.} In \cite{DaveWitteMorris}, Dave Witte Morris proved a (super)rigidity result for certain $S$-arithmetic groups. The precise statement of this result is somewhat technical, but it essentially states that an $S$-arithmetic subgroup $\mathbb{G}_{\mathcal{O}(S)} \subset \mathbb{G}$ of a solvable linear algebraic group over a number field has a closed intermediate subgroup $\mathbb{G}_{\mathcal{O}(S)}\subset H \subset \mathbb{G}$ so that $\mathbb{G}_{\mathcal{O}(S)}$ is \emph{superrigid} in $H$. This means that every continuous representation $\mathbb{G}_{\mathcal{O}(S)} \to \GL_n(\R)$ extends (after possibly restricting to a finite index subgroup and taking finite quotients of the range) to a continuous representation $H \to \GL_n(\R)$. For a precise definition of superrigidity and of the intermediate group $H$, we refer the reader to \cite{DaveWitteMorris}. Morris and Studenmund \cite{DaveDaniel} later extended this result to representations to $\GL_n(L)$ for local fields $L$ other than $\R$. 

Although $\BS(1,p^l)$ is $S$-arithmetic, these results do not contradict Theorem \ref{thm:Classification} or Corollary \ref{cor:Corollary} (which find \emph{non-superrigid} lattices) since $G_n$ does not satisfy the hypotheses of the above-mentioned theorems---$G_n$ is not solvable, and (as we show in Section \ref{sec:NotLinear}) it does not admit a faithful linear representation over a field of characteristic $0$.

\subsection*{Lattice Envelopes of $S$-arithmetic groups.} 

When $\Gamma \subset G$ is a uniform lattice, we say $G$ is an \emph{envelope} of $\Gamma$. Given a finitely generated group $\Gamma$, it is a basic question to classify the envelopes of $\Gamma$. Remarkably, a result of Bader-Furman-Sauer \cite{LatticeEnvelopes} shows that this classification is possible for a wide class of group if one is willing to work up to \textit{virtual isomorphism} (essentially up to finite index subgroups and compact quotients; see \cite{LatticeEnvelopes} for details).

\begin{theorem}{\cite[Theorem B]{LatticeEnvelopes}}
\label{thm:LatticeEnvelopes}
    Let $\Gamma \subset H$ be a lattice embedding.
    \begin{enumerate}
        \item If $H$ is a center-free, real semisimple Lie group $H$ without compact factors that is not locally isomorphic to $\SL_2(\R)$, then every nontrivial lattice embedding of $\Gamma$ into a locally compact second countable (lcsc) group $G$ is virtually isomorphic to $\Gamma \subset H$.
        \item Let $H$ be a connected, noncommutative, absolutely simple adjoint $\Q$-group\footnote{This result actually applies to arbitrary number fields and to larger sets of places. See \cite{LatticeEnvelopes} for a complete statement.} and set $\Gamma = H(\Z[1/p])$. Then every nontrivial lattice embedding of $\Gamma$ into a lcsc group is virtually isomorphic to $\Gamma \subset H(\R)\times H(\Q_p)$ or to a tree extension $\Gamma \subset H^\ast$.
    \end{enumerate}
\end{theorem}

\noindent A \textit{tree extension} is a particular type of intermediate subgroup $H^\ast$ $$\Gamma  = H(\Z[1/p]) \subset H^\ast \subset H(\R) \times \Aut(T_{1,p}),$$ where $T_{1,p}$ denotes the Bruhat-Tits tree. 

The lattice embedding of $\BS(1,p^2)$ in Equation \ref{Eqn:EmbeddingOfBS1p} is into a solvable group, hence does not satisfy the hypotheses of Theorem \ref{thm:LatticeEnvelopes}. Nevertheless, the group $G_n$ considered in this paper is a tree extension. Proposition \ref{prop:IsometryClassification} gives a precise description of $G_n$ as a subgroup of $B(\R) \times \Aut(T_{1,n})$. We do not currently have a description of the virtual isomorphism types of the embeddings $\phi_{1,m}$.

\subsection*{Incommensurable Baumslag-Solitar Lattices.} Forester \cite{Forester} and Verma \cite{Verma} constructed and studied incommensurable uniform lattices in the group of automorphisms of a combinatorial model, $X_{m,n}$, of the Baumslag-Solitar groups $\BS(m,n)$. Although their model spaces $X_{m,n}$ are superficially similar to the space $X_n$ studied in this paper, their requirement that automorphisms preserve the cell structure of $X_{m,n}$ make the groups $\Aut( X_{m,n})$ and $\Isom(X_n)$ quite different---for instance, Forester's $\Aut(X_{m,n})$ contain no \textit{pure translations}, which act trivially on the tree but not on $X_n$ (see Section \ref{sec:XnPrelim}).

\section{The space $X_n$ and its group of isometries}
\label{sec:XnPrelim}

In this section we will construct the space $X_n$ and compute its group of isometries. We actually give two constructions: one as the universal cover of a piecewise-Riemannian cell complex, and one as a horocyclic product of the hyperbolic plane and a tree. 

\subsection*{Constructions of $X_n$.}
\label{subsec:ConstructionOfXn}

Let $\Ha^2 = \{(x,y) \in \R^2 \mid y > 0\}$ denote the hyperbolic plane in the upper half plane model with the standard Riemannian metric $\frac{1}{y^2}(dx^2 +dy^2)$. Consider the ``horostrip" $Y = \{(x,y) \mid 1 \leq y \leq n\}\subset \Ha^2$. The map $\tau:(x,y) \to (x+n,y)$ is an isometry of $Y$, and $Y/\tau$ is a topological cylinder. The boundary of $Y/\tau$ is two circles, one with length $1$, and the other with length $n$. Form the space $C_n$ by identifying the two circles with the unique degree $n$ locally isometric map sending $[(0,n)]\to [(0,1)]$. The resulting space is topologically the mapping torus of $S^1 \overset{n}{\to}S^1$, so $\pi_1(C_n) \cong \BS(1,n)$. Let $X_n$ be the universal cover of $C_n$ with the induced piecewise-Riemannian metric. 

\begin{figure}[H]
    \centering

\begin{tikzpicture}[scale = .9]

\fill[fill={mygrey}] (0,1) -- (0,2) -- (2,2) -- (2,1) -- cycle;
    \draw  (-3,0) -- (4,0);
    \draw (0,0) -- (0,3.5);
    \draw[dotted] (-3,1) -- (4,1);
    \draw[dotted] (-3,2) -- (4,2);
    \node at (-2,3)[]{$\mathbb{H}^2$};

    \tikzset{->-/.style={decoration={
  markings,
  mark=at position #1 with {\arrow[scale=1]{>}}},postaction={decorate}}}

 \draw[-] (0,1) to  (1,1);
 \draw[-] (1,1) to  (2,1);
 \draw[-] (0,2) to  (2,2);
 \draw[->-=.6] (0,1) to  (0,2);
 \draw[->-=.5] (0,1) to  (0,2);
\draw[->-=.6] (2,1) to  (2,2);
 \draw[->-=.5] (2,1) to  (2,2);

 
\end{tikzpicture}
\hspace{2cm}
\begin{tikzpicture}[scale=.8]

\begin{scope}
    \clip (-2,.28) rectangle (2,-.85);
     \draw (0,0) ellipse (2cm and .5cm);

\end{scope}
     \draw[dashed] (0,0) ellipse (2cm and .5cm);
    \draw (0,2) ellipse (1cm and .25cm);
     \draw[->] (2.25,.5)   to[bend right] node[midway,above right] {$\times 2$} (1.5,1.75);

     \draw (-1,2) to [bend left] (-2,0);
     \draw (1,2) to [bend right] (2,0);

\end{tikzpicture}

    \caption{\small{(left) The horostrip $Y$, the fundamental domain of $(x,y) \to (x+n,y)$, and (right) the identification of the boundary of $Y$ the topological cylinder. }}
    \label{fig:Horostrip}
\end{figure}
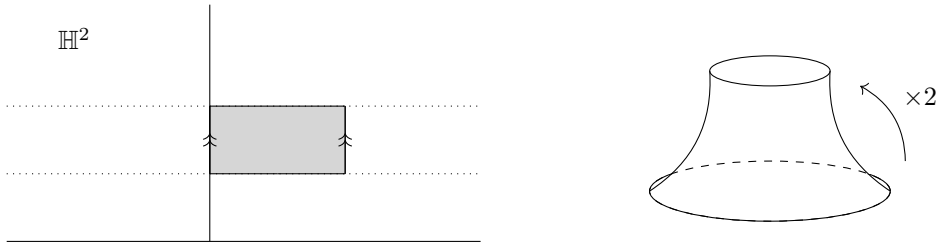

The above description of $X_n$ is its quickest definition. For our analysis, we will require a more explicit construction of both $X_n$ and of the action of $\BS(1,n)$. As we will shortly describe, $X_n$ is a \emph{horocyclic product} of $\Ha^2$, and $T_{1,n}$, the homogeneous oriented tree where each vertex has $1$ incoming and $n$ outgoing edges. We will think of the edges oriented upward. For brevity, we will denote $T_{1,n}$ by $T$. Let $\Aut(T)$ denote the group of orientation-preserving automorphisms of $T$ (by which we mean automorphisms which preserve the orientation of edges). With the compact open topology, $\Aut(T)$ is a locally compact topological group. For $v \in T$, the stabilizer of $v$ is the inverse limit of finite groups $$\Stab_{\Aut(T)}(v) = \mathop{\lim_{\longleftarrow}}_{r} \Stab_{\Aut(B_r(v))}(v)$$ (where $B_r(v)$ denotes the ball of radius $r$ in $T$) and is therefore compact.

Pick a basepoint $v_0 \in T$. We will now define \textit{height functions} for $\Ha^2$ and $T$. For $(x,y)\in \Ha^2$, set $h(x,y) = \log_n(y)$. For $w \in T$ (possibly not a vertex), let $h(w)$ be the signed length of the geodesic connecting $v_0$ to $w$ (so that traversing an edge with the orientation contributes $1$, and against the orientation contributes $-1$). The horocyclic product, $X_n$, of $\Ha^2$ and $T$ is defined to be the pullback of the two height functions:

\[\begin{tikzcd}
	{X_n} & {T_{1,n}} \\
	{\mathbb{H}^2} & {\mathbb{R}}
	\arrow["\pi", from=1-1, to=1-2]
	\arrow["h", from=1-2, to=2-2]
	\arrow["h"', from=2-1, to=2-2]
	\arrow["p"', from=1-1, to=2-1]
\end{tikzcd}\]

Explicitly, $X_n = \{(x,v) \in \Ha^2 \times T \mid h(x) = h(v)\}$. Above each point $w\in T$, there is a \emph{horocycle} $\{((x,y),w)\in \Ha^2 \times T \mid y = n^{h(w)}\}$. If $\ell \subset T$ is an oriented line, then $\pi\inv(\ell) = \{((x,y), w) \mid \log_n(y) = h(w), \; w \in \ell\}$ can be identified with the hyperbolic plane. Then $X_n$ is a union of copies of $\Ha^2$ (one for each oriented line $\ell \subset T$) glued along horocycles; see Figure \ref{fig:X_n}. We then give $X_n$ a piecewise-Riemannian structure by declaring that sections $\sigma:\Ha^2 \to X_n$ of $p$ be isometric embeddings. Specifically, we say a path $\gamma:[a,b]\to X_n$ is smooth if the projection $p\circ \gamma$ is smooth, in which case we define its length to be $\int_a^b \|(p\circ \gamma)'(t)\|_{\Ha^2}dt$. The length of a piecewise smooth path is the sum of the lengths of the pieces. The distance between two points in $X_n$ is the infimum of the lengths of all piecewise smooth paths connecting the two points.

 \begin{figure}[H]
     \centering
    \begin{tikzpicture}[scale=.6]
    \draw (-2.588,-4.326) -- (-2.588,-2.982);
    \draw (-2.588,-4.326) -- (2.494,-3.623);
    \draw (2.494,-3.623) -- (2.485,-2.304);
    \draw (-2.588,-2.982) -- (-3.684,-1.214);
    \draw (-2.588,-2.982) -- (-1.95,-1.378);
    \draw (-1.95,-1.378) -- (3.127,-.683);
    \draw (3.127,-.683) -- (2.485,-2.304);
    \draw (-2.588,-2.982) -- (2.485,-2.304);
    \draw (-1.95,-1.378) -- (-1.747,0.484);
    \draw (-1.95,-1.378) -- (-0.19,0.28);
    \draw (-3.684,-1.214) -- (-1.91,-0.979);
    \draw (-3.684,-1.214) -- (-3.754,1.17);
    \draw (-3.684,-1.214) -- (-5.3,1.35);
    \draw (-1.747,0.484) -- (-1.176,2.29);
    \draw (-1.747,0.484) -- (-2.185,2.61);
    \draw (-1.747,0.484) -- (-0.163,0.692);
    \draw (-0.026,2.054) -- (-0.19,0.28);
    \draw (-0.19,0.28) -- (1.691,1.761);
    \draw (-0.19,0.28) -- (4.886,0.966);
    \draw (4.886,0.966) -- (3.127,-.683);
    \draw (-5.3,1.35) -- (-5.934,3.332);
    \draw (-5.3,1.35) -- (-6.946,3.664);
    \draw (-5.3,1.35) -- (-3.955,1.537);
    \draw (-3.754,1.17) -- (-4.78,3.094);
    \draw (-3.754,1.17) -- (-3.43,2.864);
    \draw (-1.88,4.36) -- (-6.946,3.664);
    \draw (-1.58,3.93) -- (-1.88,4.36);
    \draw (-0.86,4.02) -- (-5.934,3.332);
    \draw (0.28,3.796) -- (-4.78,3.094);
    \draw (-1.94,1.41) -- (-3.754,1.17);
    \draw (0.28,3.796) -- (-4.78,3.094);
    \draw (0.28,3.796) -- (0.488,3.41);
    \draw (-3.43,2.864) -- (1.65,3.565);
    \draw (-2.185,2.61) -- (2.88,3.313);
    \draw (-1.176,2.29) -- (3.894,2.974);
    \draw (-0.026,2.054) -- (5.05,2.75);
    \draw (1.691,1.761) -- (6.75,2.45);
    \draw (6.75,2.45) -- (4.886,0.966);
    \draw (-0.86,4.02) -- (-0.742,3.666);
    \draw (1.65,3.565) -- (1.577,3.139);
    \draw (2.88,3.313) -- (2.978,2.857);
    \draw (3.894,2.974) -- (3.774,2.578);
    \draw (5.05,2.75) -- (5,2.22);

    \node at (1.629,-.888)[circle,fill,inner sep=.75pt]{};
    \node at (-.757,-1.215)[circle,fill,inner sep=.75pt]{};
    \node at (.988,-2.504)[circle,fill,inner sep=.75pt]{};
     \node at (-.204,-2.663)[circle,fill,inner sep=.75pt]{};
    \node at (-1.396,-2.823)[circle,fill,inner sep=.75pt]{};

    \draw [dashed] (-1.396,-2.823) -- (-.757,-1.215);
    \draw [dashed] (.988,-2.504) -- (1.629,-.888);

\end{tikzpicture}
     \caption{\small{A piece of $X_2$, with two horobricks marked.}}
     \label{fig:X_n}
 \end{figure}
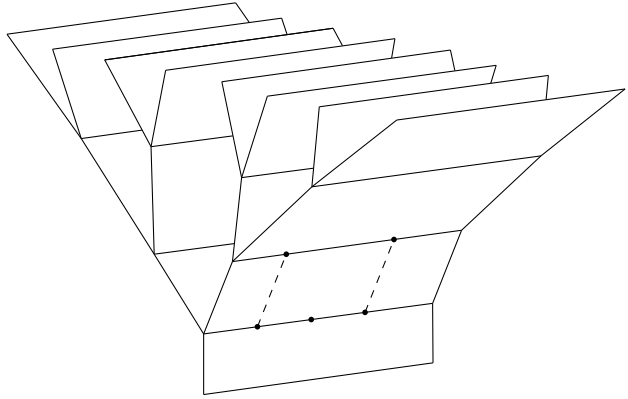

\subsection*{The action of $\BS(1,n)$ on $X_n$.}
\label{subsec:ActionOfBS}

It is not immediately obvious why the two constructions of $X_n$ given above are equivalent. In this section, we will construct the action of $\BS(1,n)$ on $X_n$, which we will later show to be a covering space action, with base space $C_n$. Our first step is to define actions of $\BS(1,n)$ on $\Ha^2$ and $T$, as summarized by \cite[Proposition 1.1]{FarbMosherII}, stated below.

\begin{proposition}[Farb-Mosher]
\label{prop:ActionOnXn}
There exist actions of $\BS(1,n)$ on $X_n$, $\Ha^2$, $T$, and $\R$ which makes the maps in the pullback diagram of $X_n$ equivariant.
\end{proposition}

This is proven in \cite{FarbMosherII} by first defining an affine action of $\BS(1,n)$ on the $n$-adic rational numbers $\Q_n$, and identifying $T$ with the Bruhat-Tits building of $\Q_n$. What follows is essentially just an explicit unwrapping of the previous sentence. We will require this explicit form in Section \ref{sec:ClassificationOfEmbeddings}.

\begin{proof}
    The actions of $\BS(1,n)$ on $\R$ and $\Ha^2$ are given by $$\begin{array}{l}
a:x\to x\\
b: x \to x+1
\end{array} \qquad \text{and}\qquad  \begin{array}{l}
a:(x,y) \to (x+1,y)\\
b:(x,y) \to (nx,ny)
\end{array}$$ respectively. Note that $h(a\cdot x) = h(x)$, and $h(b\cdot x) = h(x) + 1$ for all $x \in \Ha^2$.

We will now describe the action on $T$. For a vertex $v \in T$, let $\up(v) \subset T$ be the subtree consisting of vertices $w$ so that the geodesic connecting $v$ to $w$ is coherently oriented. For $k \in \N$, let $$\up_k(v) = \{w \in \up(v) \mid d(v,w) = k\}.$$ Give  $\up_1(v_0)$ a bijective labeling with $\Z/n$. Inductively, give $\up_k(v_0)$ a bijective labeling with $\Z/n^k$ subject to the condition that the label of $w \in \up_k(v_0)$ is congruent to the label of its downward neighbor modulo $n^{k-1}$. See Figure \ref{fig:LabeldUp}. For each $k$, let $\sigma_k = (0 \; 1 \; \cdots \; n^k-1) \in S_{\Z/n^k}$ be the full cycle given by $x \to x+1$. The congruence condition on the labels ensures that the collection of permutations $(\sigma_k)_{k \in \N}$ defines a graph isomorphism of $\up(v_0)$. Call this isomorphism $a$. We will shortly extend it to an element of $\Aut(T)$.

\begin{figure}
    \centering
\begin{tikzpicture}
  [level distance=16mm,
   level 1/.style={sibling distance=53mm},
   level 2/.style={sibling distance=18mm},
   level 3/.style={sibling distance=6mm},
   grow'=up,scale=.9]
  \coordinate
     child  {
       child {
            child { node[label={above:$0$}] {} } 
            child { node[label={above:$9$}] {} } 
            child { node[label={above:$18$}] {}} 
            node[label={[align=left]left:$0$}] {}}
       child {
            child { node[label={above:$3$}] {} } 
            child { node[label={above:$12$}] {} } 
            child { node[label={above:$21$}] {}} 
            node[label={[align=left]left:$3$}] {}}
        child {
            child { node[label={above:$6$}] {} } 
            child { node[label={above:$15$}] {} } 
            child { node[label={above:$24$}] {}} 
            node[label={[align=left]left:$6$}] {}}
     node[label={[align=left]left:$0$}] {}}
     child {
       child {
            child { node[label={above:$1$}] {} } 
            child { node[label={above:$10$}] {} } 
            child { node[label={above:$19$}] {}} 
            node[label={[align=left]left:$1$}] {}}
       child {
            child { node[label={above:$4$}] {} } 
            child { node[label={above:$13$}] {} } 
            child { node[label={above:$22$}] {}} 
            node[label={[align=left]left:$4$}] {}}
        child {
            child { node[label={above:$7$}] {} } 
            child { node[label={above:$16$}] {} } 
            child { node[label={above:$25$}] {}} 
            node[label={[align=left]left:$7$}] {}}
     node[label={[align=left]left:$1$}] {}}
     child {
       child {
            child { node[label={above:$2$}] {} } 
            child { node[label={above:$11$}] {} } 
            child { node[label={above:$20$}] {}} 
            node[label={[align=left]left:$2$}] {}}
       child {
            child { node[label={above:$5$}] {} } 
            child { node[label={above:$14$}] {} } 
            child { node[label={above:$23$}] {}} 
            node[label={[align=left]left:$5$}] {}}
        child {
            child { node[label={above:$8$}] {} } 
            child { node[label={above:$17$}] {} } 
            child { node[label={above:$26$}] {}} 
            node[label={[align=left]left:$8$}] {}}
     node[label={[align=right]right:$2$}] {}};
\node [label=left:{$v_0$}] at (0,0)[circle,fill,inner sep=2pt]{};
\end{tikzpicture}

    \caption{\small{A labeling of the subtree $\up(v_0)$. The action of $a \in \BS(1,n)$ on $w \in \up(v_0)$ is given by adding $1$ to each label, modulo $n^{h(w)}$.}}
    \label{fig:LabeldUp}
\end{figure}
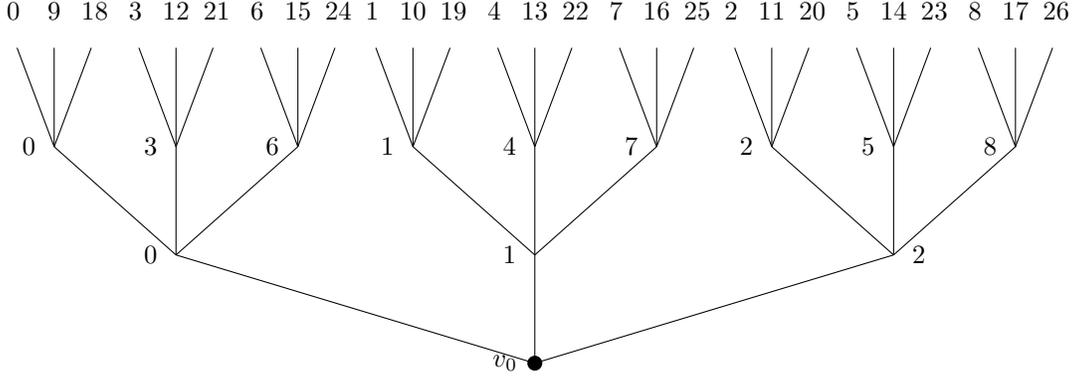

The maps $\Z/n^k \to \Z/n^{k+1}$ given by $[x] \to [nx]$ give an injective (but not surjective) graph homomorphism $b:\up(v_0) \to \up(v_0)$. When $w \in \up(v_0)$ has label divisible by $n$, (so that $b\inv(v)$ is defined) the arithmetic relation $n(x/n + 1) = x +n$ gives the Baumslag-Solitar relation $bab\inv(v) = a^n(v)$. Extend $b:\up(v_0) \to \up(v_0)$ to a map $T \to T$ arbitrarily (this is possible since $T$ is homogeneous). The following fact will be used frequently: for every $x \in T$, there is some $k$ so that $b^k \cdot x \in \up(v_0)$. To see this, let $\gamma$ be the path connecting $v_0$ to $x$. If $x \not \in \up(v_0)$, $\gamma$ will move downward in the tree (as measured by the height function) before moving up to $x$. Let $k$ be the number of vertices $\gamma$ passes through before moving upward. Then $b^k\cdot x$ lies in $\up(v_0)$, proving the fact.

We will now extend $a:\up(v_0) \to \up(v_0)$ to the entire tree. For $x \in T$, let $k$ be such that $b^k\cdot x \in \up(v_0)$. Then set $a \cdot x = b^{-k} (a|_{\up(v_0)})^{n^k}b^k \cdot x$. This does not depend on $k$ since $bab\inv = a^n$ on $b\cdot \up(v_0)$. The two elements $a,b \in \Aut(T)$ satisfy $bab\inv \cdot x = a^n \cdot x$ for \textit{any} $x \in T$, giving an action of $\BS(1,n)$ on $T$. Note that $h(a\cdot v) = h(v)$, and $h(b\cdot v) = h(v) + 1$ for all $v \in T$, just like the action on $\Ha^2$.

We then define the action of $g \in \BS(1,n)$ on $X_n = \{(x,v) \in \Ha^2 \times T \mid h(x) = h(v)\}$ by $g\cdot (x,v) = (g_{\Ha^2}\cdot x, g_{T}\cdot v)$, where the actions on $\Ha^2$ and $T$ are defined as above. This is well-defined since $h(x) = h(v)$ implies $h(g\cdot x) = h(g\cdot v)$. It is an action by isometries since for any smooth curve $\gamma:[a,b] \to X_n$, we have $$\|(g_{X_n}\circ \gamma)'(t)\|_{X_n} = \|(p \circ g_{X_n}\circ \gamma)'(t)\|_{\Ha^2} = \|(g_{\Ha^2} \circ p\circ \gamma)'(t)\|_{\Ha^2} = \|(p\circ \gamma)'(0)\|_{\Ha^2} = \|\gamma'(t)\|_{X_n},$$ so $g_{X_n}$ preserves the lengths of all curves.

\end{proof}

After studying the isometries of $X_n$ in more detail, we will show that this action of $\BS(1,n)$ on $X_n$ is a covering space action with quotient isometric to $C_n$. This will show the equivalence of the two constructions of $X_n$.

\subsection*{The group of isometries, $\Isom(X_n)$.}

In the previous subsection we constructed isometries of $X_n$ by specifying an isometry of $\Ha^2$ and an orientation-preserving automorphism of $T$. In this subsection we make that correspondence precise. For an isometry $f$ of $\Ha^2$ fixing $\infty$ in the upper half plane model (resp. an orientation-preserving isometry of $T$) define the \emph{height change} $h(f)$ by $h(f\cdot x) - h(x)$, for $x \in \Ha^2$ (resp. $x \in T$). Note that if $\partial f(\infty) = \infty$ (resp. $f$ is orientation-preserving), this quantity does not depend on $x$. Throughout this paper, $\Aut(T)$ will denote the group of orientation-preserving automorphisms of $T$. 

\begin{proposition}
    \label{prop:IsometryClassification}
    Let $$G_n^{\pm} = \{(f,g) \in \Isom(\Ha^2) \times \Aut(T) \mid \partial f(\infty) = \infty, \; h(f) = h(g)\}.$$ Then $\Isom(X_n) \cong G_n^{\pm}$.  
\end{proposition}

\begin{proof}
    The map $G_n^{\pm}\to \Isom(X_n)$ is constructed as in the previous section: $(f,g)\to [(x,v) \to (f(x),g(v))]$. This is well-defined since $h(f) = h(g)$. It is an isometry by the computation from the previous section. 

    To construct the inverse map,  let $f \in \Isom(X_n)$, and $s:\Ha^2\to X_n$ be a section of $p$. Then $p\circ f \circ s: \Ha^2 \to \Ha^2$ is an isometry. Since any two oriented lines $\ell_1,\ell_2 \subset T$ intersect below some common vertex $v$, any other section $s'$ agrees with $s$ below some horocycle. Then $p_\ast(f) \coloneq p\circ f \circ s = p\circ f \circ s'$ does not depend on the choice of $s$.

    For every $v \in T$, we call $\pi\inv(v)$ a horocycle, and if $v$ is a vertex, we say it is a branching horocycle. Since (branching) horocycles must be sent to (branching) horocycles, there is an induced orientation-preserving map on $T$. Let $\pi_\ast(f) \in \Aut(T)$ denote this map. The map $f \to (p_\ast(f), \pi_\ast(f))$ is the desired inverse.
    
\end{proof}

The group of isometries of $\Ha^2$ fixing the point at infinity in the upper half plane model is isomorphic to $\Aff(\R)\cong \R \rtimes \R^\ast$, and contains an index two subgroup of orientation-preserving isometries. In this paper, we deal primarily with the group of orientation-preserving isometries of $X_n$, by which we mean those isometries with orientation-preserving $\Aff(\R)$ component. We denote this group by $G_n = \Isom^+(X_n)$.

\subsection*{Classification of isometries of $X_n$: hyperbolic and elliptic elements.}

Every automorphism $f$ of a tree is either \emph{elliptic}: $\Fix(f)\neq \emptyset$, or \emph{hyperbolic}: there is an $f$-invariant bi-infinite geodesic $\Axis(f) \subset T$ on which $f$ acts by a nontrivial translation (see \cite[Section 6.4]{SerreTrees}). We will use a similar classification scheme for $G_n$. An element $f\in G_n$ is called \emph{elliptic} (resp. \emph{hyperbolic}) if the associated action $\pi_\ast(f)$ on $T$ is elliptic (resp. hyperbolic). Note that being elliptic and hyperbolic are mutually exclusive conjugacy invariants in $G_n$. Also note that $\Fix(gfg\inv) = g\cdot \Fix(f)$ and $\Axis(gfg\inv) = g\cdot \Axis(f)$ when $f$ is elliptic and hyperbolic respectively. 

\subsubsection*{Translation distance.} If $f \in G_n$ is elliptic, it has height change $0$. Then $p_\ast(f) \in \Stab_{\Isom(\Ha^2)}(\infty)$ also preserves height and therefore acts by a translation in the upper half plane. Fix once and for all an identification $$\Stab_{\Isom(\Ha^2)}(\infty) \cong \Aff(\R) = \R \rtimes \R^\ast$$ so that the translation $(x,y) \to (x+t,y)$ in the upper half plane is sent to the translation $x \to x+ t$ of $\R$ (that is, $(t,1) \in \R \rtimes \R^\ast)$). This allows us to associate  to each elliptic element $f\in G_n$ a real number, which we call the \textit{translation distance}, denoted $\td(f)$. This quantity is equivariant in the sense that $$\td(gfg\inv) = n^{h(g)}\cdot \td(f).$$

\subsection*{Pure translations and pure tree actions.} We wish to bring attention to two particular subgroups of $G_n$. First are the \emph{pure translations}, the subgroup acting trivially on $T$ (ie the kernel of $\pi_\ast$). The subgroup of pure translations form the connected component of the identity in $G_n$ and is isomorphic to $\R$. Second, there is a group of \emph{pure tree actions,} defined as follows. Let $h:\Aut(T) \to \Z$ denote the height change homomorphism, and for $r \in \R^\ast$, let $d(r) \in \Aff(\R)$ denote the dilation $x \to rx$. Then $f:\Aut(T) \to G_n$ given by $g \to (d(n^{h(g)}), g) \in G_n$ is a splitting of $$1\to \R \to G_n \overset{\pi_\ast}{\to} \Aut(T) \to 1.$$ Hence, $G_n \cong \R \rtimes_h \Aut(T)$. We will call the image of $f$ the group of \textit{pure tree actions} on $X_n$. 

\subsection*{$\mathbf{\BS(1,n)} \curvearrowright X_n$ is a covering space action.}

We will now show that the action of $\BS(1,n)$ on $X_n$ defined in Proposition \ref{prop:ActionOnXn} is a covering space action. Let $q \in X_n$ be any point, and $U$ a small ball around $q$ (radius less than $\ln(2)/2$ will suffice). Let $g \in \BS(1,n)$ be such that $g\cdot U \cap U \neq \emptyset$. Since $\BS(1,n)$ acts transitively on edges of $T$, we can assume (by conjugating $g$) that the vertex immediately under $\pi(q)$ is $v_0$. Using the normal form for $\BS(1,n)$ (see \cite[Lemma 1]{CollinsAutBS}) every $g \in \BS(1,n)$ can be written $g = b^{-x}a^yb^z$ for $x,y,z \in \Z$ and $x,z \geq 0$. Since height change is valued in $\Z$, and $U$ is small, we have $h(g) = z-x = 0$. A direct computation shows that the only elements of the form $b^{-x}a^yb^x$ fixing $v_0$ satisfy $y \mid n^x$. Then $g = a^{y'}$ for some $y' \in \Z$. Since $g\cdot U \cap U \neq \emptyset$, $U$ is small and $\td(a^{y'}) \in \Z$, we conclude that $g$ is trivial, and that $\BS(1,n) \curvearrowright X_n$ is a covering space action. 

Let $\ell \subset T$ be the line segment connecting $v_0$ and $v_1$. Translates of the horobrick $H = \{((x,y),w) \in X_n \mid x \in [0,n] \; y \in [1,n], \; w \in \ell \}$ cover all of $X_n$, since $\pi\inv(\ell) = \cup_{k \in \Z} (a^n)^k\cdot H$ (the ``horostrip"), and $\BS(1,n) \cdot \ell = T$. An adaptation of the previous paragraph shows that $H \cap g\cdot \mathring{H} \neq \emptyset$ implies $g = 1$, so $H$ is a fundamental domain. It remains to compute the identifications of the boundary. The elements $a,a^2,\ldots, a^{n-1}$ account for the identifications of the lower edges in Figure \ref{fig:identificationsDisplayed}, $a^n$ identifies the two vertical edges, and $b$ identifies a lower edges with the upper edge. Any other $g = b^{-x}a^yb^z \in \BS(1,n)$ makes $H \cap gH$ either empty or a corner (which we already know are identified). Then these are the only identifications on the boundary of the horobrick, so $X_n / \BS(1,n)$ is isometric to $C_n$. Then the horocyclic product construction of $X_n$ is the same as the universal cover of $C_n$.

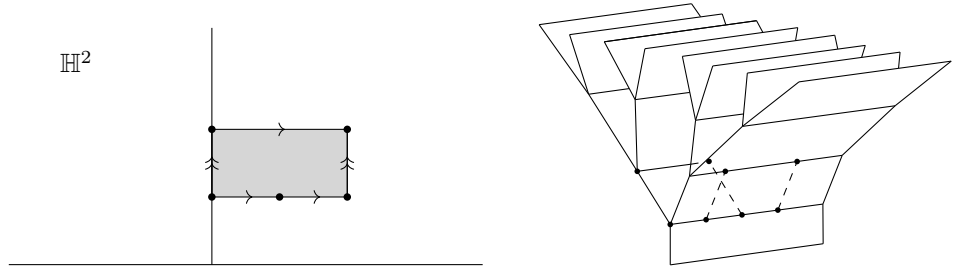
\begin{figure}[H]
    \centering
    \begin{tikzpicture}[scale = .9]

\fill[fill={mygrey}] (0,1) -- (0,2) -- (2,2) -- (2,1) -- cycle;
    \draw  (-3,0) -- (4,0);
    \draw (0,0) -- (0,3.5);
    \node at (0,1)[circle,fill,inner sep=1pt]{};
    \node at (0,2)[circle,fill,inner sep=1pt]{};
    \node at (1,1)[circle,fill,inner sep=1pt]{};
    \node at (2,1)[circle,fill,inner sep=1pt]{};
    \node at (2,2)[circle,fill,inner sep=1pt]{};
    \node at (-2,3)[]{$\mathbb{H}^2$};

    \tikzset{->-/.style={decoration={
  markings,
  mark=at position #1 with {\arrow[scale=1]{>}}},postaction={decorate}}}

 \draw[->-=.6] (0,1) to  (1,1);
 \draw[->-=.6] (1,1) to  (2,1);
 \draw[->-=.54] (0,2) to  (2,2);
 \draw[->-=.6] (0,1) to  (0,2);
 \draw[->-=.5] (0,1) to  (0,2);
\draw[->-=.6] (2,1) to  (2,2);
 \draw[->-=.5] (2,1) to  (2,2);
 
\end{tikzpicture}
\hspace{.5cm}
\begin{tikzpicture}[scale=.4]
    \draw (-2.588,-4.326) -- (-2.588,-2.982);
    \draw (-2.588,-4.326) -- (2.494,-3.623);
    \draw (2.494,-3.623) -- (2.485,-2.304);
    \draw (-2.588,-2.982) -- (-3.684,-1.214);
    \draw (-2.588,-2.982) -- (-1.95,-1.378);
    \draw (-1.95,-1.378) -- (3.127,-.683);
    \draw (3.127,-.683) -- (2.485,-2.304);
    \draw (-2.588,-2.982) -- (2.485,-2.304);
    \draw (-1.95,-1.378) -- (-1.747,0.484);
    \draw (-1.95,-1.378) -- (-0.19,0.28);
    \draw (-3.684,-1.214) -- (-1.91,-0.979);
    \draw (-3.684,-1.214) -- (-3.754,1.17);
    \draw (-3.684,-1.214) -- (-5.3,1.35);
    \draw (-1.747,0.484) -- (-1.176,2.29);
    \draw (-1.747,0.484) -- (-2.185,2.61);
    \draw (-1.747,0.484) -- (-0.163,0.692);
    \draw (-0.026,2.054) -- (-0.19,0.28);
    \draw (-0.19,0.28) -- (1.691,1.761);
    \draw (-0.19,0.28) -- (4.886,0.966);
    \draw (4.886,0.966) -- (3.127,-.683);
    \draw (-5.3,1.35) -- (-5.934,3.332);
    \draw (-5.3,1.35) -- (-6.946,3.664);
    \draw (-5.3,1.35) -- (-3.955,1.537);
    \draw (-3.754,1.17) -- (-4.78,3.094);
    \draw (-3.754,1.17) -- (-3.43,2.864);
    \draw (-1.88,4.36) -- (-6.946,3.664);
    \draw (-1.58,3.93) -- (-1.88,4.36);
    \draw (-0.86,4.02) -- (-5.934,3.332);
    \draw (0.28,3.796) -- (-4.78,3.094);
    \draw (-1.94,1.41) -- (-3.754,1.17);
    \draw (0.28,3.796) -- (-4.78,3.094);
    \draw (0.28,3.796) -- (0.488,3.41);
    \draw (-3.43,2.864) -- (1.65,3.565);
    \draw (-2.185,2.61) -- (2.88,3.313);
    \draw (-1.176,2.29) -- (3.894,2.974);
    \draw (-0.026,2.054) -- (5.05,2.75);
    \draw (1.691,1.761) -- (6.75,2.45);
    \draw (6.75,2.45) -- (4.886,0.966);
    \draw (-0.86,4.02) -- (-0.742,3.666);
    \draw (1.65,3.565) -- (1.577,3.139);
    \draw (2.88,3.313) -- (2.978,2.857);
    \draw (3.894,2.974) -- (3.774,2.578);
    \draw (5.05,2.75) -- (5,2.22);

    \node at (-1.3,-0.883)[circle,fill,inner sep=.75pt]{};
    \node at (1.629,-.888)[circle,fill,inner sep=.75pt]{};
    \node at (-.757,-1.215)[circle,fill,inner sep=.75pt]{};
    \node at (.988,-2.504)[circle,fill,inner sep=.75pt]{};
    \node at (-.204,-2.663)[circle,fill,inner sep=.75pt]{};
    \node at (-1.396,-2.823)[circle,fill,inner sep=.75pt]{};
    \node at (-3.684,-1.214)[circle,fill,inner sep=.75pt]{};
    \node at (-2.588,-2.982)[circle,fill,inner sep=.75pt]{};

    \draw [dashed] (-1.396,-2.823) -- (-.757,-1.215);
    \draw [dashed] (.988,-2.504) -- (1.629,-.888);
    \draw [dashed] (-.204,-2.663) -- (-1.3,-0.883);
    \draw [dashed] (-1.91,-0.979) -- (-1.3,-0.883);
    
\end{tikzpicture}
    \caption{\small{(left) Identifications of the horobrick defining $C_n$. (right) the horobrick $H$, and its translate $aH$.}}
    \label{fig:identificationsDisplayed}
\end{figure}

\section{Lattices in $G_n$ are uniform and isomorphic to $\BS(1,n^l)$.}
\label{sec:LatticesAreUniform}

In this section, we will prove Theorem \ref{thm:LatticesAreUniformAndBS}, which says that all lattices in $G_n$ are uniform and isomorphic to $\BS(1,n^l)$. The proof has three main steps: 

\begin{enumerate}
    \item Show that every lattice $\Gamma \subset G_n$ contains a hyperbolic element.
    \item Using the hyperbolic element found in part 1, find a compact subset $K \subset G_n$ so that the restriction of $G_n \to G_n/\Gamma$ to $K$ is surjective.
    \item Show that $\Gamma$ is torsion free, and quote a result of Farb-Mosher \cite{FarbMosherII}.
\end{enumerate}

Step 1 is accomplished by finding an explicit fundamental domain for a discrete group $\Gamma$. This gives a formula for the Haar measure of the quotient $G_n/\Gamma$, which we then show must be infinite if $\Gamma$ contains no hyperbolic elements. The compact set $K\subset G_n$ found in step 2 is, loosely speaking, a ``rectangle" in $G_n$, whose height is determined by the hyperbolic element found in part 1, and whose width is determined by an elliptic element (which must also exist). For step 3, we use a result of Farb-Mosher, which says that a torsion free group quasi-isometric to $\BS(1,n)$ is isomorphic to some $\BS(1,k)$. If $\gamma \in \Gamma$ has finite order, conjugating by a well-chosen hyperbolic element gives infinitely many elements of a (compact) point-stabilizer, contradicting the discretness of $\Gamma$.

\subsection*{Step 1: $\Gamma$ contains a hyperbolic element}

Lemma \ref{lem:HaarFormula} (below) gives a formula for the covolume of a discrete subgroup $\Gamma \subset G_n$. We will first explain the notation of Equation \ref{eqn:HaarFormula} and show it is well-formed expression.

Let $\Gamma\subset G_n$ be discrete, and let $v \in T$. The discreteness of $\Gamma$ implies that the point stabilizer $\Gamma_v$ has a discrete set of translation distances $\td(\Gamma_v) = \{|\td(g)| \mid g\in \Gamma_v\}$: if $\td(\Gamma_v)\subset \R$ were not discrete, there would be a sequence of isometries $\gamma_i = (f_i,g_i) \in \Gamma_v$ with $(f_i)_{i \in \N} \subset \Isom(\Ha^2)$ converging to the identity. Since $\Stab_{\Aut(T)}(v)$ is compact, the $g_i$ subconverge, so the $(f_i,g_i)$ subconverge, contradicting the discreteness of $\Gamma$. Similarly, discreteness also implies that the subgroup $\Gamma_{v,0} = \{g \in \Gamma_v \mid \td(g) = 0\}$ is finite. Let $a_v = \min (\td(\Gamma_v)\setminus \{0\})$. 

We will now show that for $[v] \in T/\Gamma$, the quantity $a_v\cdot n^{h(v)}/|\Gamma_{v,0}|$ (which appears in Lemma \ref{lem:HaarFormula} below) does not depend on the choice of representative of $[v]$. Let $w \in [v]$, so that there is some $\gamma \in \Gamma$ with $\gamma\cdot v = w$. Then $\gamma \Gamma_v\gamma\inv = \Gamma_w$, and $\gamma \Gamma_{v,0}\gamma\inv = \Gamma_{w,0}$. In particular, $|\Gamma_{v,0}| = |\Gamma_{w,0}|$. The equivariance of translation distance gives $a_w = n^{h(\gamma)} a_v= n^{h(w)-h(v)}a_v$, and hence $a_wn^{-h(w)} = a_vn^{-h(v)}$. Then $a_v\cdot n^{h(v)}/|\Gamma_{v,0}| = a_w\cdot n^{h(w)}/|\Gamma_{w,0}|$, as claimed.

\subsubsection*{The Haar measure.} $G_n = \Isom(X_n)$ is a locally compact topological group (with the compact open topology), and hence has a left-invariant Haar measure $\mu$. Fix a basepoint $x_0 = ((0,1),v_0) \in X_n$. For $Y \subset X_n$, let $F(Y) \coloneq \{g\in G_n \mid g(x_0)\in Y\}$. When $Y$ is open $F(Y)$ is open. Since $X_n$ has compact point stabilizers, if $Y$ is compact then $F(Y)$ is compact. For $v\in T$ and $t \in \R$, consider the sets $Y_{v,t} = \{((s,n^{h(v)}),v) \mid s \in [0,t)\}$. The set $\overline{F(Y_{v_0,1})}$ is compact with non-empty interior. It therefore has non-zero finite Haar measure. Normalize $\mu$ so that $\mu(F(Y_{v_0,1})) = 1$. We will now compute the Haar measure of $G_n/\Gamma$.

\begin{lemma}
    \label{lem:HaarFormula}
    Let $\Gamma \subset G_n$ be discrete and $\mu$ the Haar measure of $G_n$ with the above normalization. Then 
    \begin{equation}\label{eqn:HaarFormula}\mu_{G_n/\Gamma}(G_n/\Gamma) = \sum_{[v]\in T/\Gamma} \frac{a_v\cdot n^{-h(v)}}{|\Gamma_{v,0}|}.\end{equation}
\end{lemma}

Formula \ref{eqn:HaarFormula} and its proof are an adaptation of the formula for the covolume of a tree lattice, see \cite[Section 1.5]{TreeLattices}. 

\begin{proof}

We will explicitly construct a fundamental domain for the action of $\Gamma$ on $G_n$ as a disjoint union $\cup_{[v]\in T/\Gamma} S_v$ with $\mu(S_v) = a_v\cdot n^{h(v)}/|\Gamma_{v,0}|$. The sets $S_v$ will be found in the course of the proof.

For each $v\in T$, consider the set $$F(Y_{v,a_v}) = \{g \in G_n \mid g(x_0) = ((t,n^{h(v)}),v)\; \text{for}\; t \in [0,a_v)\}.$$ We first claim that $\mu(F(Y_{v,a_v})) = a_v\cdot n^{-h(v)}$. For any $q \in \N$, left-translates of $F(Y_{v_0,1/q})$ by pure translations of length $1/q$ form a disjoint cover $F(Y_{v_0,1})$. Then $\mu(F(Y_{v_0,1/q})) = 1/q$. For any $t >0$, we can approximate $F(Y_{v_0,t})$ by a union of translates of $F(Y_{v_0,1/q})$ for small $q$. Regularity of $\mu$ then gives that $\mu(F(Y_{v_0,t})) = t$. Let $v \in T$, and let $g\in G_n$ be a hyperbolic element taking $((0,n^{h(v)}),v) \in X_n$ to $((0,1),v_0)$. Then $g\cdot Y_{v,a_v} = Y_{v_0,n^{-h(v)}\cdot a_v}$, so $$\mu(F(Y_{v,a_v})) = \mu(g\cdot F(Y_{v,a_v})) = \mu(F(Y_{v_0,n^{-h(v)}\cdot a_v})) = n^{-h(v)}\cdot a_v,$$ proving the claim.

For each $v \in T$, the finite group $\Gamma_{v,0}$ acts on the finite measure space $F(Y_{v,a_v})$ freely. We will construct a fundamental domain, $S_v$, of this action. Around each $p \in F(Y_{v,a_v})$, there is an open set $U$ containing $p$ so that $\gamma \cdot U\cap U  = \emptyset $ for all $\gamma \in \Gamma_{v,0}$. Choose countably many such $(p_i, U_i)$ so that $\{U_i\}_{ i \in \N}$ cover $F(Y_{v,a_v})$. Then set $$S_v = \cup_{j = 1}^\infty (U_j \setminus \cup_{i = 1}^{j-1} \Gamma_{v,0}\cdot U_i).$$ This is a fundamental domain by construction, so $\mu(S_v) = a_v\cdot n^{-h(v)}/|\Gamma_{v,0}|$, as desired.

For every $ [v] \in T/\Gamma$, choose some representative $v \in [v]$. We now show that $\cup_{[v]\in T/\Gamma} S_v$ is a fundamental domain for the action of $\Gamma$. Every $f \in G_n$ sends $v_0$ to some $w \in [v] \in T/\Gamma$. Then there is some $\gamma_1 \in \Gamma$ so that $\gamma_1 f$ sends $v_0$ to $v$, and some $\gamma_2 \in \Gamma_v$ so that $\gamma_2\gamma_1f \in F(Y_{v,a_v})$, and finally some $\gamma_3 \in \Gamma_{v,0}$ so that $\gamma_3\gamma_2\gamma_1f \in S_v$.  

If $v,w\in T$ and $[v] \neq [w]$, no element of $\Gamma$ can bring an element from $S_v$ into $S_w$. Similarly, no element of $\Gamma$ can identify two elements of $S_v$, since such a $\gamma$ would necessarily have $\td(\gamma) < a_v$, hence $\gamma \in \Gamma_{v,0}$. But $S_v$ is a fundamental domain for the action of $\Gamma_{v,0}$ on $F(Y_{v,a_v}) \supset S_v$. This shows that $\cup_{[v]\in T/\Gamma} S_v$ is a fundamental domain for $\Gamma$, as required.

\end{proof}

We are now prepared to prove that every lattice $\Gamma \subset G_n$ must have a hyperbolic element.

\begin{lemma}
    If $\Gamma \subset G_n$ is discrete and has no hyperbolic elements, then the covolume $\mu_{G_n/\Gamma}(G_n/\Gamma) $ is infinite.
\end{lemma}

\begin{proof}
    We will exhibit an infinite sequence $T_1,T_2,\ldots \subset T/\Gamma$ of disjoint subsets of $T/\Gamma$ with $\sum_{[v] \in T_i} \frac{a_vn^{-h(v)}}{|\Gamma_{v,0}|}$ uniformly bounded away from $0$.

    Pick any $v \in T$. If $\Gamma_{v,0}$ does not act trivially on $\up(v)$, there is some $w \in \up(v)$ with $|\Gamma_{w,0}| < |\Gamma_{v,0}|$. Since each $|\Gamma_{v,0}|$ is finite, we can iterate this process to find some $v$ so that $\Gamma_{v,0}$ acts trivially on $\up(v)$. 

    Let $i \in \N$, and set $T_i = \up_i(v) = \{w \in \up(v) \mid d(v,w)=i\}$. Note that the $T_i/\Gamma$ are pairwise disjoint, since $\Gamma$ contains no hyperbolic elements. We will show that $$\sum_{[w] \in \up_i(v)/\Gamma} \frac{a_wn^{-h(w)}}{|\Gamma_{w,0}|} = \frac{a_v}{|\Gamma_{v,0}|},$$ which suffices to prove divergence. The basic idea is that as we move up one level in the tree, the extra factor of $n^{-1}$ is balanced exactly by the extra vertices (with weight $a_w$) at that level.

    Let $\gamma \in \Gamma_v$ realize $a_v$, so that $\td(\gamma) = a_v$. Since $\Gamma_{v,0} \triangleleft \Gamma_v$ and $a_v$ is minimal, there is a semidirect decomposition $\Gamma_v = \Gamma_{v,0} \rtimes \langle \gamma \rangle$.

    Let $w \in \up_i(v)$. Since $\Gamma_w \subset \Gamma_v$ and $a_v$ is minimal, there is some $k \in \N$ so that $a_v\cdot k = a_w$. Let $k' = |\orb_{\gamma}(w)|$, the exponent of the smallest power of $\gamma$ fixing $w$. Any element of $\Gamma_w$ realizing $a_w$ must take the form $\gamma^kg$, for $g \in \Gamma_{v,0}$ and fix $w$. Then $w = \gamma^k g \cdot w = \gamma^k \cdot w$ (recall that $\Gamma_{v,0}$ acts trivially on $\up(v)$), so $k' \leq k$. But $\gamma^{k'}$ fixes $w$, so the minimality of $a_w$ gives $a_w \leq a_v\cdot k'$, hence $a_v \cdot k \leq a_v \cdot k'$. This shows $k = k'$, that is, $a_w = a_v \cdot |\orb_\gamma(w)|$. 

    Any element sending $w \in \up(v)$ to another $w' \in \up(v)$ must fix $v$. Then $\up_i(v)/ \Gamma = \up_i(v)/\Gamma_v $, so

\begin{align*}
    \sum_{[w]\in \up_i(v)/\Gamma} \frac{a_wn^{-h(w)}}{|\Gamma_{w,0}|} &= \sum_{[w]\in \up_i(v)/\Gamma_{v}} \frac{a_wn^{-i}}{|\Gamma_{v,0}|}\\
    &= \frac{1}{n^i \cdot |\Gamma_{v,0}|} \sum_{[w] \in \up_i(v)/\Gamma_{v}} |\orb_\gamma(w)|\cdot a_v \\
    &= \frac{a_v}{n^i \cdot |\Gamma_{v,0}|} \cdot |\up_i(v)|\\
    &= \frac{a_v}{|\Gamma_{v,0}|}.
\end{align*}

This shows that the sum diverges.

\end{proof}

This completes Step 1 of the proof of Theorem \ref{thm:LatticesAreUniformAndBS}: every lattice $\Gamma \subset G_n$ contains a hyperbolic element.

\subsection*{Step 2: $\Gamma$ is uniform.}

Let $t \in \Gamma$ denote the hyperbolic element found in the previous step. If $\Gamma$ contains no elliptic elements, then the quotient $G_n \to G_n/\Gamma$ is injective on $F(Y_{v_0,r})$, which has (arbitrarily large) measure $r$. Then $\Gamma$ must contain an elliptic element $s$ with positive translation distance. 

The next lemma describes how the discreteness of $\Gamma$ limits how an elliptic element $s$ can act on $T$ in the presence of a hyperbolic element $t$ with $h(t)>0$. This setup is quite constraining because the conjugates $t^{-k} s t^k$ have exponentially decreasing translation distance---if $\Gamma$ is to be discrete, these small translations must be balanced by large motion in the tree. The tree is only so large, and the lemma says that the only way to match the exponentially decreasing translation distances is by (after taking a sufficiently large power of $s$, and restricting to a subtree) acting by larger and larger full cycles (just like how $a \in \BS(1,n)$ acts on $\up(v_0)$). This observation also constitutes an important step in the proof of Theorem \ref{thm:Classification}. It is the technical core of the paper.

\begin{lemma}
\label{lem:EventuallyTransForever}
    Let $\Gamma \subset G_n$ be discrete. Let $t \in \Gamma$ be a hyperbolic element with positive height change $h = h(t) > 0$, and let $s \in \Gamma$ be elliptic with $\td(s)>0$.  Let $w_0 \in \Axis(t)\cap \Fix(s) \subset T$, and for $i \in \Z$, set $w_i = t^i\cdot w_0$. Then there exists a $k \in \N$ so that if $s^j$ is the smallest power of $s$ fixing $w_k$, then for all $i \in \N$, the action of $s^j$ on $\up_i(w_k)$ is transitive. 
    
\end{lemma}

When $s \in \Stab(v)$ acts transitively on $\up_i(v)$ for each $i \in \N$ (as $s^j$ does in the conclusion of Lemma \ref{lem:EventuallyTransForever}), we will say $s$ \emph{acts transitively forever on} $\up(v)$. This concept will be used frequently in the remainder of the paper.

\begin{figure}
    \centering
    \begin{tikzpicture}

\begin{scope}
  [level distance=10mm,level/.style={sibling distance=15mm/#1},grow'=up]
  \coordinate
    child foreach \x in {0,1}
      {child foreach \y in {0,1}
        {child foreach \z in {0,1}}};
\end{scope}

\node[label=left:{$w_k$}] at (0,0)[circle,fill,inner sep=1.5pt]{};

\node[label=below:{$w_0$}] at (1.5,-1.2)[circle,fill,inner sep=1.5pt]{};

\node[rotate=0] at (2.5,-.5){$\iddots$};

\node[rotate=0] at (.5,-.5){$\ddots$};




\node[label=left:{$1$}] at (-.75,10mm)[circle,fill,inner sep=1.5pt]{};

\node[label=left:{$1$}] at (-1.125,20mm)[circle,fill,inner sep=1.5pt]{};

\node[] at (-1.375,30mm)[circle,fill,inner sep=1.5pt]{};

\node[label=left:{$w_{k+1}$}] at (-1.5,10mm)[]{};

\node[label=left:{$w_{k+2}$}] at (-1.5,20mm)[]{};

\node[label=left:{$w_{k+3}$}] at (-1.5,30mm)[]{};


\draw[dashed] (-1.6,-.2) rectangle (1.5,3.15);


\node[label=left:{$2$}] at (.75,10mm)[]{};

\node[label=left:{$2$}] at (-.3,20mm)[]{};

\node[label=left:{$3$}] at (.5,20mm)[]{};

\node[label=left:{$4$}] at (1.2,20mm)[]{};

\node[label=above:{$1$}] at (-1.375,30mm)[]{};

\node[label=above:{$2$}] at (-.9,30mm)[]{};

\node[label=above:{$3$}] at (-.6,30mm)[]{};

\node[label=above:{$4$}] at (-.2,30mm)[]{};

\node[label=above:{$5$}] at (.2,30mm)[]{};

\node[label=above:{$6$}] at (.6,30mm)[]{};

\node[label=above:{$7$}] at (.9,30mm)[]{};

\node[label=above:{$8$}] at (1.375,30mm)[]{};

\node[] at (-1.375,30mm)[circle,fill,inner sep=1.5pt]{};


\node[label=right:{$s^j|_{\up_1(w_k)} = (1 \; 2)$}] at (1.5,10mm)[]{};

\node[label=right:{$s^j|_{\up_2(w_k)} = (1 \; 3\; 2 \; 4 )$}] at (1.5,20mm)[]{};

\node[label=right:{$s^j|_{\up_3(w_k)} = (1 \; 5\;3\;7\;2\;6\;4\;8)$}] at (1.5,30mm)[]{};

\end{tikzpicture}

    \caption{Lemma \ref{lem:EventuallyTransForever} guarantees that there is some $w_k = t^k\cdot w_0$ so that some power of $s$ acts on $\up(w_k)$ in the same way that $a$ acts on $\up(v_0)$. See also Figure \ref{fig:LabeldUp} in Section \ref{sec:XnPrelim}.}
    \label{fig:EventuallyTransForever}
\end{figure}

\begin{proof}

If $i > i'$ and $s^j$ acts transitively on $\up_i(w_k)$, then it also acts transitively on $\up_{i'}(w_k)$. To prove transitivity on $\up_i(w_k)$, we can then ``round up" to a multiple of $h$, and assume $h \mid i$. This shows that the conclusion of Lemma \ref{lem:EventuallyTransForever} is equivalent to the following: there exists a $k \in \N$ so that if $s^j$ is the smallest power of $s$ fixing $w_k$, then for all $l > k$, the exponent of the smallest power of $s$ fixing $w_l$ is $jn^{h(l-k)}$.

Assume that such a $k$ does not exist. Then there is an increasing sequence $l_i \in \N$ so that if $s^{j_i}$ is the smallest power of $s$ fixing $w_{l_i}$, then $j_i < j_{i-1} n^{h(l_i - l_{i-1})}$. We will show that $g_i = t^{-l_i} s^{j_i} t^{l_i}$ subconverges, contradicting the discreteness of $\Gamma$. The translation distances $$0 < \td(g_i) =\frac{j_i\td(s)}{n^{l_i}} < \frac{j_{i-1}\td(s)}{n^{l_{i-1}}} = \td(g_{i-1})$$ are bounded from below and decreasing, hence subconverge. Since $g_i \cdot w_0 = w_0$ and $\Stab_{\Aut(T)}(w_0)$ is compact, any subsequence of $(\pi_\ast(g_i))_{i \in \N}$ subconverges. Let $(g_{a_i})_{i \in \N}$ be a subsequence of $(g_i)_{i \in \N}$ so that both $(\td(g_{a_i}))_{i \in \N}$ and $(\pi_\ast(g_{a(i)}))_{i \in \N}$ converge. Since $g_{a_i} \in \Gamma$ is elliptic, we can write $$g_{a_i} = (\td(g_{a_i}), \pi_\ast(g_{a_i})) \in \R \rtimes \Aut(T) \cong G_n,$$ which converges. This contradicts the discreteness of $\Gamma$.   
    
\end{proof}

We are now ready to prove that all lattices $\Gamma \subset G_n$ are uniform.

\begin{proof}[Proof of Theorem \ref{thm:LatticesAreUniformAndBS}, part 1.]
Let $\Gamma \subset G_n$ be a lattice. By Step 1, $\Gamma$ contains a hyperbolic element $t \in \Gamma$ hyperbolic and elliptic element $s \in \Gamma$. Assume without loss of generality that $h = h(t) >0$ and $\td(s)>0$. Let $w_0 \in \Axis(t)\cap \Fix(s)$, and $w_i = t^i\cdot w_0$. From Lemma \ref{lem:EventuallyTransForever}, there is some $k$ so that if $s^j$ is the smallest power of $s$ fixing $w_k$, then for all $i \in \N$, the action of $s^j$ on $\up_{i}(w_k)$ is transitive.
    
Let $\ell \subset T$ be the geodesic connecting $w_0$ and $w_1$. We show that $\Gamma \cdot \ell = T$. Let $v \in T$. There is some $i$ so that $t^i\cdot v \in \up(w_k)$. Since $s^j$ acts transitively on $\up_m(w_k)$ for all $m$, there is some power of $s$ so that $s^pt^i\cdot v \in \Axis(t)$. Further application of $t\inv$ then shows that $\gamma \cdot v \in \ell$ for some $\gamma \in \Gamma$, proving the claim.

Let $a_{w_1} = \min (|\td(\Gamma_{w_1})| \setminus 0)$, and set $$R = \{((x,y),v) \in X_n\mid x \in [0,a_{w_1}]\; v \in \ell \}.$$ Then the $\Gamma$ translates of $K = F(R) = \{g\in G_n \mid g(x_0) \in R\}$ cover $G_n$. Since $R$ is compact, so is $K$, and hence $\Gamma$ is uniform.

\end{proof}

\subsection*{Step 3: $\Gamma$ is torsion free.}

Let $\Gamma \subset G_n$ be a lattice. We will now show $\Gamma \cong \BS(1,n^l)$ for some $l$. The proof begins by showing $\Gamma$ is torsion free.

\begin{lemma}
    Let $\Gamma \subset G_n$ be a lattice. Then $\Gamma$ is torsion free.
\end{lemma}

\begin{proof}

Let $\Gamma \subset G_n$ be a lattice and let $t,s\in \Gamma$ be hyperbolic and elliptic elements with $h(t)>0$ and $\td(s)>0$ (which exist by Step 1). Let $\gamma \in \Gamma$ be nontrivial with finite order. Then $\gamma$ is elliptic, stabilizing some $v \in T$. If $\Axis(t) \not \subset \Fix(\gamma)$ then the elements $\gamma, t\gamma t\inv, t^2 \gamma t^{-2},\ldots \in \Stab(v)$ are all distinct, since the heights $h(\max(\Axis(t) \cap \Fix(t^k \gamma t^{-k})))$ are all distinct. Each $t^k \gamma t^{-k}$ has finite order, so $\td(t^k \gamma t^{-k}) = 0$. thus there are infinitely many elements in $\Gamma_{v,0}$, which is impossible since $\Gamma$ is discrete.

It thus suffices to find some finite order $\delta \in \Gamma$ with $\Axis(t) \not \subset \Fix(\delta)$. Since $\gamma$ is nontrivial with finite order, there is some $x \in T$ with $\gamma \cdot x \neq x$. If $x \in \Axis(t)$, we are done. If not, let $g \in \langle t,s \rangle$ so that $g\cdot x \in \Axis(t)$. Such a $g$ exists because the $\Gamma$ translates of the segment $\ell \subset \Axis(b)$ from the proof of Theorem \ref{thm:LatticesAreUniformAndBS}, part 1 cover $T$. Then $g\gamma g\inv$ is a finite order element with $\Axis(t) \not \subset \Fix(g\gamma g\inv)$. This completes the proof that $\Gamma$ is torsion free.

\end{proof}

Since $\Gamma$ is torsion free, we can apply \cite[Corollary 5.3]{FarbMosherII}, reproduced below.

\begin{proposition}[Farb-Mosher]
    Let $G$ be a finitely generated, torsion free group which is quasi-isometric to $\BS(1,n)$, for some $n \geq 2$. Then $G \cong \BS(1,k)$ for some $k \geq 2$ such that $\BS(1,k)$ is abstractly commensurable with $\BS(1,n)$.
\end{proposition}

\begin{proof}[Proof of Theorem \ref{thm:LatticesAreUniformAndBS} part 2.]

Let $\Gamma \subset G_n$ be a lattice. In Step 2, we found a compact set $R \subset X_n$ so that $\Gamma \cdot R = X_n$. discreteness of $\Gamma$ implies that the action of $\Gamma$ on $X_n$ is properly discontinuous. By the Milnor-Schwartz lemma, $\Gamma$ is quasi-isometric to $X_n$, which is quasi-isometric to $\BS(1,n)$. Since $\Gamma$ is torsion free, we can apply the Farb-Mosher result \cite[Corollary 5.3]{FarbMosherII} above, and conclude that $\Gamma \cong \BS(1,k)$ is abstractly commensurable to $\BS(1,n)$. We will now show that $k = n^l$ for some $l$. 

Let $c,d \in \Gamma$ satisfy $dcd\inv = c^k$ Since height change is a conjugacy invariant, $$h(c) = h(c^k) = kh(c),$$ so $h(c) = 0$ and $c$ is elliptic. If $\td(c) = 0$, the compactness of point-stabilizers would imply the sequence $c,c^2,c^3,\ldots $ subconverges, contradicting the discreteness of $\Gamma$. Then $\td(c) \neq 0$. We also have $$n^{h(d)} \cdot \td(c) = \td(dcd\inv) = \td(c^k) = k\cdot \td(c)$$, so $k = n^l$ for some $l$, as desired. 
    
\end{proof}

\section{Classification of lattice embeddings $\BS(1,n^l) \to G_n$}
\label{sec:ClassificationOfEmbeddings}

In this section we prove Theorem \ref{thm:Classification}, which gives a classification of lattice embeddings $\psi:\BS(1,n^l) \to G_n$. For readability and ease of notation, we will first prove the case $l = 1$, and address the general case once we have proven the supporting lemmas in the $l = 1$ case.

Let $c = \psi(a)$, $d = \psi(b)$. By slight abuse of notation, also let $a,b \in G_n$ denote the standard action of $\BS(1,n)$ on $X_n$ (as defined in Section \ref{subsec:ActionOfBS}). There are six steps to the proof.

\subsection*{Step 1: $d$ is hyperbolic and $c$ is elliptic.}  

Since height change is a conjugacy invariant, $$h(c) = h(dcd\inv) = h(c^n) = n\cdot h(c),$$ so $h(c) = 0$, meaning $c$ is elliptic. If $\td(c) = 0$, then $c,c^2,c^3,\ldots $ subconverges, contradicting the discreteness of $\Gamma = \phi(\BS(1,n))$. Then $\td(c) \neq 0$. We also have $n^{h(d)}\cdot \td(c) = \td(dcd\inv) = \td(c^n) = n\cdot \td(c)$, so $$h(d) = 1.$$ 

\subsection*{Step 2: line up $d$ and $b$.}

Let $g \in G_n$ be any element satisfying $g\cdot \Axis(d) = \Axis(b)$. By conjugating $\psi$ by $g$, we can assume $\Axis(d) = \Axis(b)$. The following lemma shows that the $\Aut(T)$ components of $b$ and $d$ are actually conjugate, and parameterizes the conjugating elements.

\begin{lemma}
    \label{lem:HyperbolicConjugation}
    Let $b,b' \in \Aut(T)$ be hyperbolic automorphisms of $T$ with the same axis, and height change $1$. Let $v_0 \in \Axis(b)$, and let $T_0$ be the maximal subtree of $T$ containing $v_0$, but no other elements of $\Axis(b)$. Then there is a bijection $$B:\{g \in \Aut(T) \mid b = gb'g\inv, \; g\cdot v_0 = v_0\} \to \Aut(T_0,v_0),$$ where $\Aut(T_0,v_0)$ denotes the automorphisms of the rooted tree $(T_0,v_0)$. 
\end{lemma}

\begin{proof}
    If $g \in G_n$ satisfies $g\cdot v_0 = v_0$ and $gb'g\inv = b$, then $g$ preserves $\Axis(b) = \Axis(b')$, and therefore leaves $T_0$ invariant. Define $B(g) \coloneq g|_{T_0}$. 
    
    We will now construct the inverse of $B$. Let $T_k = b^k\cdot T_0$. Note that $\Axis(b) \cup (\cup_k T_k) = T$. Given $g_0 \in \Aut(T_0,v_0)$, define $g_k \in \Aut(T_k,v_k)$ inductively by $g_k = b g_{k-1}(b')\inv$ for $k > 0$ and by $g_k = b\inv g_{k+1} b'$ for $k < 0$ (See figure \ref{fig:hyperbolicConjugation}). The element $g\in \Aut(T)$ defined by $g(v) = g_k(v)$ for $v \in T_k$ satisfies $b = gb'g\inv$ and $g\cdot v_0 = v_0$. 
    
    We now show that $C(g_0) \coloneq g$ is the inverse of $B$. It is clear that $B \circ C = \text{id}$. For $g \in G_n$ satisfying $b = gb'g\inv$ and $g\cdot v_0 = v_0$, the two maps $C \circ B(g)$ and $g$ agree on $T_0$ by construction. The relation $b = gb'g\inv$ and induction show that they must also agree on every $T_k$, and hence that $C \circ B = \text{id}$.
\end{proof}

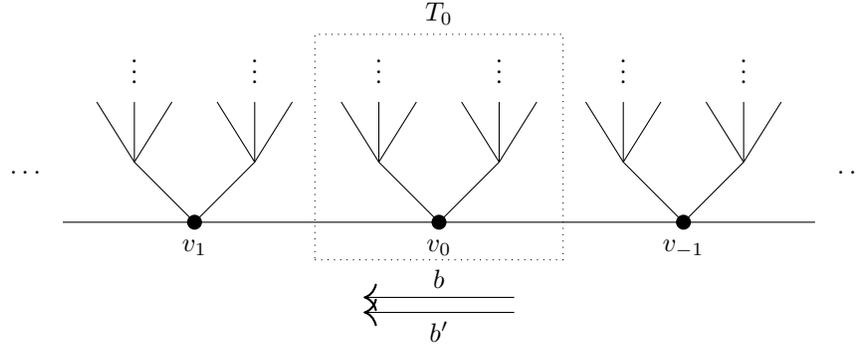
\begin{figure}
    \centering
    \begin{tikzpicture}

\draw (-5,0) -- (5,0);

\begin{scope}
[level distance=8mm,
   level 1/.style={sibling distance=16mm},
   level 2/.style={sibling distance=5mm},
   grow'=up]
  \coordinate
     child {
       child {}
       child {}
       child {}
     }
     child {
       child {}
       child {}
       child {}
     };
\end{scope}

\begin{scope}
[level distance=8mm,
   level 1/.style={sibling distance=16mm},
   level 2/.style={sibling distance=5mm},
   grow'=up,xshift=3.25cm]
  \coordinate
     child {
       child {}
       child {}
       child {}
     }
     child {
       child {}
       child {}
       child {}
     };
\end{scope}

\begin{scope}
[level distance=8mm,
   level 1/.style={sibling distance=16mm},
   level 2/.style={sibling distance=5mm},
   grow'=up,xshift=-3.25cm]
  \coordinate
     child {
       child {}
       child {}
       child {}
     }
     child {
       child {}
       child {}
       child {}
     };
\end{scope}

\node [label=below:{$v_0$}] at (0,0)[circle,fill,inner sep=2pt]{};

\node [label=below:{$v_{-1}$}] at (3.25,0)[circle,fill,inner sep=2pt]{};

\node [label=below:{$v_1$}] at (-3.25,0)[circle,fill,inner sep=2pt]{};

 \tikzset{->/.style={decoration={
  markings,
  mark=at position 1 with {\arrow[scale=2]{>}}},postaction={decorate}}}

\draw [->] (1,-1) -> (-1,-1) node[midway,above] {$b$};
\draw [->] (1,-1.2) -> (-1,-1.2)node[midway,below] {$b'$};

\node [label = above:{$\vdots$}] at (.8,1.6) {};
\node [label = above:{$\vdots$}] at (-.8,1.6) {};

\node [label = above:{$\vdots$}] at (-2.45,1.6) {};
\node [label = above:{$\vdots$}] at (-4.05,1.6) {};

\node [label = above:{$\vdots$}] at (4.05,1.6) {};
\node [label = above:{$\vdots$}] at (2.45,1.6) {};

\node [label = above:{$\ldots$}] at (5.5,.4) {};
\node [label = above:{$\ldots$}] at (-5.5,.4) {};

\draw [dotted] (-1.65,-.5) rectangle (1.65,2.5);

\node[above] at (0,2.5){$T_0$};

\end{tikzpicture}
    \caption{\small{The tree ($n = 3$) with $\Axis(b) = \Axis(b')$ displayed horizontally, and $T_0$ labeled.} }
    \label{fig:hyperbolicConjugation}
\end{figure}

Since $\pi_\ast:G_n\to \Aut(T)$ has a section, we can use Lemma \ref{lem:HyperbolicConjugation} to conjugate $\psi$ so that $\pi_\ast(d) = \pi_\ast(b)$. The $\Aff(\R)$ components of $b$ and $d$ must both have height change $1$, hence take the form $x \to nx + \alpha$. All such affine transformations are conjugate through a (pure) translation. Then $b$ and $d$ are conjugate in $G_n$. We will proceed assuming $d = b$.

\subsection*{Step 3: the case where $c$ acts transitively forever on $\up(v_0)$.}

Lemma \ref{lem:EventuallyTransForever} shows that $c$ cannot fix $\Axis(b)$ (if it did, then $b^kcb^{-k}$ would subconverge to the identity). Thus $\Axis(b)\cap \fix(c) \subset T$ has a maximum element (with respect to the poset on $T$ given by the relation $v < w$ if $w \in \up(v)$). Let $w_0 = \max(\Axis(b)\cap \fix(c))$. By conjugating $\psi$ by powers of $b$, we can assume $w_0 = v_0$. In this step, we address the case where $c$ acts transitively forever on $\up(v_0)$. That is, assume the action of $c$ on $\up_i(v_0)$ is transitive for all $i \in \N$. We will show that $\psi$ is conjugate to the standard embeddings $\phi_{1,1}:\BS(1,n) \to G_n$. That is, there is some element $g \in G_n$ so that $$g \psi(a) g\inv = gcg\inv = a \qquad \text{and} \qquad g\psi(b)g\inv = gbg\inv = b.$$

Let $i \in \N$ and $v_i = b^i\cdot v_0$. Since both $a$ and $c$ act transitively forever on $\up(v_0)$, every $w \in \up_i(v_0)$ can be expressed as $w = a^r \cdot v_i$ and as $w = c^{r'}\cdot v_i$ for some $r, r'$. Then there exist bijections $f_i: \up_i(v_0) \to \up_i(v_0)$ given by $c^r\cdot v_i \to a^r\cdot v_i$ conjugating the actions of $a$ and $c$. That is, $f c|_{\up(v_0)}f\inv = a|_{\up(v_0)}$. Two vertices $v\in \up_i(v_0)$ and $w \in \up_{i-1}(v_0)$ are connected by an edge if and only if $v = a^r \cdot v_i$ and $w = a^{r'}\cdot v_{i-1}$ for $r \equiv r' \mod{n^{i-1}}$. This is likewise true for powers of $c$, so the collection $\{f_i\}_{i \in \N}$ constitute a graph isomorphism $f: \up(v_0) \to \up(v_0)$ conjugating $c|_{\up(v_0)}$ to $a|_{\up(v_0)}$.

We will now extend $f$ to a map on the entire tree that conjugates $c$ to $a$. In order to prove the theorem in this case, the conjugating element must also commute with $b$. This element will be constructed with the help of Lemma \ref{lem:HyperbolicConjugation}. Since $f$ fixes $v_1$, it restricts to an automorphism of $T_0$. Let $g$ denote the map found in the proof of Lemma $\ref{lem:HyperbolicConjugation}$ with $g_0 = f|_{T_0}$. 

We next prove that $f(x) = g(x)$ for $x \in \up_i(v_0)$ by induction on $i$. If $x \in T_0$, then the conclusion follows immediately from the construction. In particular the base case follows. If $x \in \up_{i+1}(v_0)\setminus  T_0$, then $x = c^r\cdot v_i$ with $n \mid r$. Then $b\inv \cdot x \in \up_{i}(v_0)$, and $$g_{i+1}(x) = bg_i b\inv (x) = bf_i(b\inv c^r\cdot v_{i+1})= bf_i (c^{r/n}\cdot v_i)= b a^{r/n} \cdot v_i= a^r b \cdot v_i= a^r \cdot v_{i+1} = f_{i+1}(x).$$ This closes the induction. 

From Lemma \ref{lem:HyperbolicConjugation}, $\pi_\ast(g)$ commutes with $\pi_\ast(b)$ by construction. Since $b,g \in \Aut(T) \subset G_n$ are pure tree actions, we know that $g$ also commutes with $b$ in $G_n$. We next verify that $\pi_\ast(gcg\inv) = \pi_\ast(a)$. For $x \in \up(v_0)$, we know that $gcg\inv (x) = a(x)$ since $f(x) = g(x)$. If $x \not \in \up(v_0)$, then there is some $b^k$ so that $b^k \cdot x \in \up(v_0)$, and $$gcg\inv (x) = b^{-k} (gcg\inv)^{n^k} b^k (x) = b^{-k} a^{n^{k}} (b^k (x)) = a(x).$$

After replacing $\psi$ with $C_g\circ \psi$, we can now assume $\psi(b) = b$ and $\pi_\ast(\psi(a)) = \pi_\ast(a)$, but perhaps the $\Aff(\R)$ components $p_\ast(\psi(a))$ and $p_\ast(a)$ are different. Since $p_\ast(\psi(a))$ is elliptic, and $\pi_\ast(\psi(a)) = \pi_\ast(a)$, the isometry $\psi(a)$ is determined by its translation distance. If $\td(\psi(a)) = 0$, the embedding will fail to be both discrete and of finite covolume. Any other choice of translation distance $s \in \R \setminus \{0\}$ will give a lattice embedding (in fact, one that differs from the standard embedding by an automorphism, see Section \ref{sec:AutGn}). Then $\psi(a) = a_s$, the element specified by $\pi_\ast(a) \in \Aut(T)$ and the translation $(x,y) \to (x + s, y)$ in $\Isom(\Ha^2)$. That is, $\psi = \phi_{s,1}$. This completes the proof of Theorem \ref{thm:Classification} in the case when $c = \psi(a)$ acts transitively forever on $\up(v_0)$.

\subsection*{Step 4: $\pi_\ast(c)$ is conjugate to $A^\eta$ for some $\eta \in \Z_n$.}

Let $\eta \in \Z_n$ be an $n$-adic integer and let $\eta_i \in \Z /n^i$ be the image of $\eta $ under $\Z_n \to \Z /n^i$. Since $\eta_{i+1} \equiv \eta_i \mod{n^{i}}$, the tree automorphisms $\pi_\ast(a^{\eta_{i+1}})$ and $\pi_\ast(a^{\eta_i})$ agree on the $i$-ball $B_{i}(v_0) \subset T$. Then the sequence $\pi_\ast(a^{\eta_i})$ has a limit, which we will denote $A^\eta \in \Aut(T)$. Note in particular that $A^1 = \pi_\ast(a)$, and for $m \in \Z \subset \Z_n$ we have $A^m = \pi_\ast(a^m)$.

At this stage in the proof, we have a lattice embedding $\psi:\BS(1,n) \to G_n$ with $\psi(b) = b$ and $v_0 = \max(\Axis(b) \cap \Fix(\psi(a)))$, but with $\psi(a)$ not necessarily acting transitively forever on $\up(v_0)$. The goal of this step of the proof is to conjugate $\psi$ so that $\pi_\ast(c) = \pi_\ast(\psi(a)) = A^\eta$ for some $\eta \in \Z_n$, while keeping $\psi(b) = b$.

Lemma $\ref{lem:EventuallyTransForever}$ implies that there are $k,j \in \N$ so that $c^j$ is the smallest power of $c$ fixing $v_k$, and $c^j$ acts transitively on $\up_i(v_k)$ for all $i \in \N$. Consider the map $\rho: \BS(1,n) \to G_n$ given by $\rho(a) = b^{-k}c^j b^k$ and $\rho(b) = b$. This is a conjugate of $\psi$, and hence a lattice embedding. Since $c^j$ acts transitively forever on $v_k$, we know that $\rho(a)$ acts transitively forever on $v_0$. Applying Step 3, we conclude that there is some $g \in G_n$ commuting with $b$ and $s\in \R \setminus \{0\}$ so that $g(b^{-k}c^j b^k)g\inv = b^{-k}(gc^jg\inv) b^k = a_s$. By replacing $\psi$ with $C_g\circ \psi$, we may now assume that $b^{-k}c^jb^k = a_s$. 

The following computation shows that $a_s = b^{-k}c^jb^k$ commutes\footnote{This can also be seen by noting that the normal closure of $a$ in $\BS(1,n)$ is isomorphic to $\Z [1/n]$, an abelian group. } with $c$: $$(b^{-k}c^jb^k)c(b^{-k}c^jb^k)\inv = b^{-k}c^jc^{n^k}c^{-j}b^k = b^{-k} c^{n^k} b^k = c.$$ We are then led to investigate the elements commuting with $\pi_\ast(a) \in \Aut(T)$.

\begin{lemma}
    \label{lem:CentralizerOfa}
    Let $a \in \Aut(\up(v_0))$ be the standard action of $a \in \BS(1,n)$. Let $C_{\Aut(\up(v_0))}(a)$ denote the centralizer of $a \in \Aut(\up(v_0))$. Then the map $\Z_n \to C_{\Aut(\up(v_0))}(a)$ given by $\eta \to A^\eta|_{\up(v_0)}$ is an isomorphism.
\end{lemma}

\begin{proof}

Recall from Section \ref{subsec:ActionOfBS} that $\up_i(v_0)$ has a bijective labeling with $\Z/n^i\Z$ so that $a$ acts on $\up_i(v_0)$ by the cycle $\sigma_i = (0 \; 1 \; \cdots n^{i} -1)$ on $\up_i(v_0)$. Let $\eta = (\eta_1,\eta_2,\ldots) \in \Z_n$. Then $A^\eta$ acts by $\sigma_i^{\eta_i}$ on $\up_i(v_0)$, which clearly commutes with $\sigma_i$ and hence with $a$.

Let $f \in \Aut(\up(v_0))$ commute with $a$. Then $f$ restricts to an action on $\up_i(v_0)$ and commutes with $\sigma_i$. The centralizer of a full cycle is generated by that cycle, so $f|_{\up_i(v_0)} = \sigma_i^{\eta_i}$ for some $\eta_i$. Since $f$ is a graph automorphism, $\eta_{i+1} \equiv \eta_i \mod{n^i}$, so $\eta = (\eta_1,\eta_2,\ldots)$ specifies an element of $\Z_n$, and $f = A^\eta$. 

Finally, let $\eta = (\eta_1,\eta_2,\ldots),$ $\mu = (\mu_1,\mu_2,\ldots) \in \Z_n$ with $A^\eta = A^\mu$. Then $A^\eta$ and $A^\mu$ agree on each $\up_i(v_0)$, so $\eta_i = \mu_i$. 

\end{proof}

Since $\pi_\ast(c)|_{\up(v_0)}$ commutes with $a|_{\up(v_0)}$, we can apply the above lemma to conclude $\pi_\ast(c)|_{\up(v_0)} = A^\eta|_{\up(v_0)}$ for some $\eta$. But both $c$ and $A^\eta$ satisfy the Baumslag-Solitar relation with $b$ (ie both $bcb\inv = c^n$ and $bA^\eta b\inv = (A^\eta)^n$.) Since $\pi_\ast(c)$ and $A^\eta$ agree on $\up(v_0)$, they must in fact be equal, since every $x \in T$ can be brought into $\up(v_0)$ by some power of $b$, and $$c \cdot x = b^{-p} c^{n^p} b^p\cdot x = b^{-p} (A^\eta)^{n^p} b^p\cdot x = A^\eta \cdot x$$ for $b^p \cdot x \in \up(v_0)$. This completes Step 4.

\subsection*{Step 5: Reduce the problem to (simple) algebra in $\Z_n$.}

Step 4 shows that $\pi_\ast(c) = A^\eta$ and $a_s = b^{-k} c^j b^k$. From this we immediately read off $$s = \td(a_s)  = j\cdot n^{-k} \cdot \td(c) ,$$ so $\td(c) = \frac{sn^k}{j}$ is determined and can take any value in $\R \setminus \{0\}$. It also gives the relation $$A^{n^k} = \pi_\ast(a^{n^k}) = \pi_\ast(b^k a b^{-k}) = \pi_\ast(b^k a_s b^{-k}) = \pi_\ast(c^j) = A^{j \eta },$$ from which we conclude $ j \eta = n^k$ in $\Z_n$. Since $c^j$ is the smallest power of $c$ fixing $v_k$ and $c^{n^k}\cdot v_k= (A^\eta)^{n^k}\cdot v_k = v_k$, we know that $j \mid n^k$. Then $\eta = n^k/j$ is a solution to $j\eta = n^k$, and since no $j \in \Z \subset \Z_n$ is a zero divisor, this solution is unique. If $n \mid n^k/j$, then $\pi_\ast(c) = A^{n^k/j}$ would fix $v_1$. But $v_0$ is the maximal element of $\Axis(b)\cap \Fix(c)$, so this is impossible. Then $n \nmid n^k/j$, and $m = n^k/j$ satisfies $(\ast)$ from Theorem \ref{thm:Classification}. The observations $\td(c) = s\cdot n^k/j$ and $\pi_\ast(c) = A^{n^k/j}$ together show that $\psi = \phi_{s, m}$ takes the form described in Theorem \ref{thm:Classification}.

\subsection*{Step 6: $\mathbf{s,m}$ are conjugacy invariants.}

Steps 2 and 5 describe an algorithm for finding a kind of ``normal form" for discrete representations $\psi:\BS(1,n) \to G_n$---one first conjugates so that $\psi(b) = b$, then by powers of $b$ so that $\max(\Axis(b) \cap \Fix (c)) = v_0$, and finally by some element making $\pi_\ast(c^j) = A^{n^k}$ for some $j,k$ satisfying certain minimality conditions. Set $s(\psi) = \td(c)\cdot n^k/j$ and $m(\psi) = n^k/j$. Then $\psi = \phi_{s(\phi),m(\psi)}$. We will now show that the pair $(\psi(s),m(\psi)) $ is a well-defined conjugacy invariant of discrete embeddings of $\BS(1,n)$.

Let $\psi: \BS(1,n) \to G_n$ be discrete and let $w_0 = \max(\Axis(\psi(b)) \cap \Fix(\psi(a)))$. For $i \in \N$, let $w_i = \phi(b)^i\cdot w_0$. Then $\psi(a)$ acts on $\up(w_0)$. The values $j$ and $k$ (and therefore $m = n^k/j$) are determined from the above data in the following way: $k$ is the smallest number for which some power of $\psi(a)$ acts transitively on $\up_i(w_k)$ for all $i \in \N$, and $j$ is the smallest power of $\psi(a)$ which fixes $w_k$.

Any conjugating element preserves the above data. Let $\psi_1,\psi_2: \BS(1,n)$ be two discrete representations and set $ w_0^1 =  \max(\Axis(\psi_1(b)) \cap \Fix(\psi_1(a))) $ and $w_0^2 = \max(\Axis(\psi_2(b)) \cap \Fix(\psi_2(a)))$. If $C_g\circ \psi_1 = \psi_2$ for some $g \in G_n$, then 

\begin{enumerate}
    \item $g\cdot w_0^1 = w_0^2$ 
    \item $g\cdot \Axis(\phi_1(b)) = \Axis(\phi_2(b))$
    \item The map $g:\up(w_0^1) \to \up(w_0^2)$ conjugates the actions of $\psi_1(a)$ and $\phi_2(a)$ and preserves heights in the sense that $g\cdot \up_i(w_0^1) = \up_i(w_0^2)$ for all $i \in \N$. 
\end{enumerate}

Then the values $(j,k)$ determined from the actions of $\psi_1(a)$ and $\psi_2(a)$ on $\up(w_0^1)$ and $\up(w_0^2)$ (respectively) are the same. This shows that $m = n^k/j$ is a conjugacy invariant.\footnote{Arguably, the quantities $(j,k)$ are the natural choice of conjugacy invariant, but this makes the condition $(\ast)$ messy.}

Let $\psi:\BS(1,n) \to G_n$ be discrete and let $g_1 \in G_n$ be such that $g_1 \cdot \max(\Axis(\psi(b)) \cap \Fix(\psi(a))) = v_0$. Then $$s(\psi) = \td(g_1\psi(a)g_1\inv) \cdot n^k/j = n^{h(g_1)}\cdot \td(\psi(a)) \cdot n^k/j.$$ Let $g_2 \in G_n$, and consider $C_{g_2} \circ \psi$. Then $g_1g_2\inv\cdot \max(\Axis(C_{g_2} \circ\psi(b)) \cap \Fix(C_{g_2} \circ\psi(a))) = v_0$, so $$s(C_{g_2} \circ \psi) = \td(g_1g_2\inv(C_{g_2} \circ \psi (a))(g_1g_2\inv)\inv) \cdot n^k/j= \td(g_1\psi(a)g_1\inv) \cdot n^k/j = s(\psi),$$ meaning $s(\psi)$ is indeed a conjugacy invariant.

\vspace{.5cm}

This completes the proof of Theorem \ref{thm:Classification} for the case $l = 1$.

\subsection*{The $\mathbf{l>1}$ case.}

The proof for the $l>1$ case follows the proof of the $\BS(1,n)$ case with the following modifications.

\subsubsection*{Steps 1 and 2: straightening $\phi(b^l)$.}

Set $\phi(a) = c$ and $\phi(b^l) = d$. Then the relation $dcd\inv = c^{n^l}$ implies that $d$ is hyperbolic with height change $l$ and $c$ is elliptic with $\td(c) \neq 0$. One can find a $g \in G_n$ so that $\Axis(gdg\inv) = \Axis(b)$, and applying the below modification of Lemma \ref{lem:HyperbolicConjugation} shows that $d$ and $b^l$ are conjugate.

\begin{lemma}
    \label{lem:HyperbolicConjugationII}
    Let $b,b' \in \Aut(T)$ be hyperbolic automorphisms of $T$ with the same axis, and height change $l$. Let $s_0 \subset \Axis(b)$ be a segment of length $l-1$, and let $S_0$ be the maximal subtree of $T$ containing $s_0$, but no other elements of $\Axis(b)$. Then there is a bijection $$\{g \in \Aut(T) \mid b = gb'g\inv, \; g\cdot s_0 = s_0\} \leftrightarrow \Aut(S_0,s_0),$$ where $\Aut(S_0,s_0)$ denotes the automorphisms of the tree $S_0$ fixing $s_0$. 
\end{lemma}

\begin{proof}
    As in the proof of Lemma \ref{lem:HyperbolicConjugation}, the forward map is given by restriction, and the inverse map is constructed inductively. Set $S_k = b^k\cdot S_0$. Given $g_0 \in \Aut(S_0,s_0)$, define $g_k \in \Aut(S_k)$ by $g_k = b g_{k-1}(b')\inv$ for $k >0$ and $g_k = b\inv g_{k+1} b'$ for $k < 0$. The inverse map applied to $g_0$ is then the element found by piecing together the $g_k$. 
\end{proof}

We proceed assuming $d = b^l$. This completes steps 1 and 2. 

\subsubsection*{Step 3: $c$ acts transitively forever on $\up(v_0)$.}

Further conjugate $\psi$ by $b$ to ensure $\max(\Axis(\psi(b)) \cap \Fix(\psi(a))) = v_0$. As before, set $v_k = b^k \cdot v_0$. We proceed to analyze the action of $\psi(a) = c$ on $\up(v_0)$. 

Assume that $c$ acts transitively forever on $\up(v_0)$. Then every $w \in \up(v_0)$ can be written $w = a^r\cdot v_0 = c^{r'}\cdot v_0$. As before, we get maps $f_i:\up_i(v_0) \to \up_i(v_0)$, defined by $c^r \cdot v_i \to a^r\cdot v_i$ and a graph isomorphism $f:\up(v_0)\to \up(v_0)$ fixing $\Axis(b)$ obtained by piecing together the $f_i$. We will now extend $f$ to all of $T$ using Lemma \ref{lem:HyperbolicConjugationII} in the same way we used Lemma \ref{lem:HyperbolicConjugation} before. Let $s_0$ be the segment connecting $v_0$ and $v_{l-1}$, let $S_0$ as in Lemma \ref{lem:HyperbolicConjugationII}, let $g_0 = f|_{S_0}$, and let $g \in \Aut(T)$ be the map found in the proof of Lemma \ref{lem:HyperbolicConjugationII}. 

We show that $f(x) = g(x)$ for $x \in \up(v_0)$. If $x \in S_0$, this follows from the construction of $g$. If $x \in \up_{i}(v_0)\setminus S_0$, then $x = c^r\cdot v_i$ for $n^{l} \mid r$. Let $j = \lfloor i/l\rfloor$. Then inductively, $$g_{j}(x) = b^lg_{j-1} b^{-l} (x) = b^l f_{j-1} (b^{-l} c^r\cdot v_{i})= b^l f_{j-1} (c^{r/n^{l}}\cdot v_{i-l})= b^l a^{r/n^l} \cdot v_{i-l}= a^r b^l \cdot v_{i-l}= a^r \cdot v_{i} = f_{j}(x),$$ proving the claim. Then conjugating $\psi$ by $g$ gives $\psi(b^l) = b^l$, and $\pi_\ast(\psi(a))|_{\up(v_0)} = \pi_\ast(a)|_{\up(v_0)}$. Since both $\pi_\ast(\psi(a))$ and $\pi_\ast(a)$ satisfy the Baumslag-Solitar relation with $b^l$, and every $x \in T$ can be brought into $\up(v_0)$ by successive application of $b^l$, we conclude that $\pi_\ast(\psi(a)) = \pi_\ast(a)$, and hence that $\psi(a) = a_s$ for some $s \in \R \setminus \{0\}$. Then $\psi = \phi_{s,1}$ takes the form claimed in Theorem \ref{thm:Classification}, completing the proof in this case.

\subsubsection*{Step 4: $c$ does not act transitively forever on $\up(v_0)$.} Then Lemma \ref{lem:EventuallyTransForever} gives some $k,j$ so that $c^j$ is the smallest power of $c$ fixing $v_{lk}$, and $c^j$ acts transitively on $\up_i(v_{lk})$ for all $i \in \N$. Choose $k$ to be minimal with this property. Then we may apply step 3 to the map $\rho:\BS(1,n^l) \to G_n$ given by $\rho(a) = b^{-lk} c^j b^{lk}$ and $\rho(b^l) = b^l$. This gives a $g \in G_n$ so that $g\rho(a)g\inv = a_s$ and $gb^lg\inv = b^l$. After conjugating $\psi$ by $g$, we have $\psi(b^l) = b^l$ and $b^{-lk}c^j b^{lk} = a_s$. 

As before, $$a_sca_s\inv = (b^{-lk}c^jb^{lk})c(b^{-lk}c^jb^{lk})\inv = b^{-lk}c^jc^{n^{lk}}c^{-j}b^{lk} = b^{-lk} c^{n^{lk}} b^{lk} = c,$$ so we can use Lemma \ref{lem:CentralizerOfa} to conclude that $\pi_\ast(c)|_{\up(v_0)} = A^\eta|_{\up(v_0)}$ for some $\eta \in \Z_n$. Since both $\pi_{\ast}(c)$ and $A^\eta$ satisfy the Baumslag-Solitar relation with $\pi_\ast(b^l)$, and every $x \in T$ can be brought into $\up(v_0)$ by successive application of $b^l$, we conclude that $\pi_\ast(c) = A^\eta$.

\subsubsection*{Steps 5: Solve for $\eta$}

From $c^j = b^{lk} a_s b^{-lk}$, we have $(A^{\eta})^j = A^{n^{lk}}$, hence $j\cdot \eta = n^{lk}.$ Since $c^j$ is the smallest power of $j$ fixing $v_{lk}$, we have $j \mid n^{lk}$. Then $m = n^{lk}/j$ is the unique solution to $ j \cdot \eta = n^{lk}$. Since $k$ is minimal, $n^l \nmid j$. Since $v_0 = \max(\Axis(b) \cap \Fix(A^{n^m})$, we must also have $n \nmid m$. The relation $c^j = b^{lk} a_s b^{-lk}$ also gives $j \cdot \td(c) = n^{lk}\cdot s$. Then $\psi = \phi_{s,m}$ takes the form described in Theorem \ref{thm:Classification}.

\subsubsection*{Conjugacy invariants.} The tuple $(s,m)$ is a conjugacy invariant for the same reason as in the $\BS(1,n)$ case. Let $\psi_1,\psi_2: \BS(1,n^l) \to G_n$ be lattice embeddings. Set $ w_0^1 =  \max(\Axis(\psi_1(b)) \cap \Fix(\psi_1(a))) $ and $w_0^2 = \max(\Axis(\psi_2(b)) \cap \Fix(\psi_2(a)))$. If $C_g\circ \psi_1 = \psi_2$ for some $g \in G_n$, then $g$ conjugates the actions of $\psi_1(a)$ and $\psi_2(a)$ on $\up(w_0^1)$ and $\up(w_0^2)$ respectively. The values $j$ and $k$ (and therefore $m$) are determined by these actions, and are therefore conjugacy invariants. Similarly, the value $s$ can be computed from the translation distance of $\psi(a)$ when $\psi$ is conjugated so that $w_0 = v_0$. The invariance of this quantity then follows from the equivariance property $\td(gfg\inv) = n^{h(g)}\cdot \td(f)$ as before.

\subsection*{Lattices in the full isometry group $\mathbf{\Isom(X_n)}$.}

Let $\Gamma \subset \Isom(X_n)$ be a lattice. The subgroup $\Gamma^+\coloneq \Gamma \cap G_n$ has index $[\Gamma: \Gamma^+] \leq 2$, and is a lattice in $G_n$, hence subject to Theorems \ref{thm:LatticesAreUniformAndBS} and \ref{thm:Classification}. In particular, if $\Gamma \neq \Gamma^+$, then $\Gamma$ is a $\BS(1,n^l)$ extension of $\Z/2$ for some $l$. We will now find the possible abstract isomorphism types of $\Gamma$. 

\begin{corollary*}
    Let $\Gamma \subset \Isom(X_n)$ be a lattice. Then one of the following holds.
    \begin{enumerate}
        \item $\Gamma \subset G_n$, and therefore (by Theorem \ref{thm:LatticesAreUniformAndBS}) $\Gamma \cong \BS(1,n^l)$ for some $l \in \N$.
        \item There is an even $l \in \N$ so that $$\Gamma \cong \langle a, b,c \mid bab\inv = a^{n^l}, \; cac\inv = a^{-n^{l/2}}, \; c^2 = b  \rangle \cong \BS(1,-n^{l/2}).$$
        \item There is an $l \in \N$ and $y \in \Z$ so that $$\Gamma \cong \langle a,b,c \mid bab\inv = a^{n^l}, \; cac\inv = a\inv, \; cbc\inv = a^y b, \; c^2 = 1\rangle.$$
    \end{enumerate}
    In the final two cases, the lattice $\Gamma^+ := \Gamma \cap G_n$ is generated by $a$ and $b$ in the above presentations.
\end{corollary*}

\begin{proof}

Since $\Gamma^+ \subset G_n$ is a lattice, Theorems \ref{thm:LatticesAreUniformAndBS} and \ref{thm:Classification} show that $\Gamma^+ \cong \BS(1,n^l)$ for some $l$. By composing with an automorphism\footnote{Of course, conjugation in $G_n$ extends to an automorphism of $\Isom(X_n)$, so Theorem \ref{thm:Classification} gives that $\Gamma^+  = \langle a_s,b^l\rangle$. The automorphisms $f_r \in G_n$ constructed in the following section also extend, so we may assume it is $\langle a,b^l\rangle$. This final automorphism is not strictly necessary for the proof, but it will make notation easier.}, we may assume $\Gamma^+ = \langle a, b^l\rangle$. If $\Gamma \subset G_n$, we are done. Let $\gamma \in \Gamma \setminus \Gamma^+$. By composing with an element of $\Gamma^+  = \langle a, b^l\rangle$, we may assume $\gamma \cdot v_0$ lies on the segment of $\Axis(b)$ connecting $v_0$ and $v_{l-1}$. In particular, $0 \leq h(\gamma) < l$.

Since $\gamma \not \in G_n$, the affine component $p_\ast(\gamma)$ is given by $x \to -cx + \alpha$ for $c = n^{h(\gamma)}$ and $\alpha \in \R$. Then $\gamma a \gamma\inv \in \Gamma^+$ fixes $\gamma\cdot v_0 \in \up(v_0)$, and therefore $\gamma a \gamma\inv \in \Stab_{\Gamma^+}(v_0) = \langle a \rangle$. We can then write $\gamma a \gamma\inv = a^x$ for $x = \td(\gamma a \gamma\inv) = -n^{h(\gamma)}$.

\vspace{.3cm}

\noindent \textit{Case 1: $h(\gamma) > 0$.} Since $h(\gamma) < l$ and $\gamma ^2  \in \Gamma^+ = \langle a, b^l\rangle$, we must have $h(\gamma^2) = 2h(\gamma)  = l$. Using the normal form for $\BS(1,n^l)$, write $\gamma^2 = (b^l)^{-x} a^y (b^l)^z$ for $x,z\geq 0$, and $n^l \nmid y$ if $x,z >0$. The height constraint implies $z = x+1$. Since $\gamma \cdot v_0 \in \up(v_0)$, we can take $x = 0$, so $\gamma^2 = a^y b^l$ for some $y \in \Z$.  Then there is a a well-defined map $$H = \langle a,b,c \mid bab\inv = a^{n^l},\; cac\inv = a^{-n^{l/2}},\; c^2 = b \rangle \to \Gamma$$ given by $$a \to a \qquad b \to a^yb^l\qquad c \to \gamma. $$ The homomorphism $f:\BS(1,n^l)\to \BS(1,n^l)$ defined by $f(a) = a$ and $f(b) = a^yb$ is an automorphism of $\BS(1,n^l)$, (see \cite{CollinsAutBS}) so the above-defined map $H \to \Gamma$ restricts to an isomorphism between the subgroups $\BS(1,n^l) \cong \langle a, b\rangle \subset H$ and $\Gamma^+ \subset \Gamma$. These subgroups are both normal and have index $2$. Then the map $H \to \Gamma$ restricts to an isomorphism of index $2$ subgroups, sends $c \not \in \langle a, b\rangle$ to $\gamma \not \in \Gamma^+$, and is therefore an isomorphism. Then $$\Gamma \cong \langle a,b,c \mid bab\inv = a^{n^l}\; cac\inv = a^{-n^{l/2}}\; c^2 = b \rangle,$$ as required.

\noindent \textit{Case 2: $h(\gamma) = 0$.} Then $\gamma a \gamma\inv = a\inv$, and $p_\ast(\gamma)$ takes the form $x \to -x + m$, which is a reflection of $\R$. Furthermore, $\gamma^2 \in \Stab_{\Gamma^+}(v_0) = \langle a \rangle$, so we can write $\gamma^2 = a^x$ for $x = \td(\gamma^2) = 0$. As in the first case, we know $\gamma b^l \gamma\inv \in \Gamma^+$, so we can use the normal form to write $\gamma c \gamma\inv = (b^l)^{-x}a^y (b^l)^z$ for some $x,z \geq 0$, and $n^l \nmid y$ if $x,z > 0$. Since $\gamma b^l \gamma\inv \cdot v_0 \in \up(v_0)$ we conclude $x = 0$ and $z = 1$ as before. Then $\gamma c \gamma\inv = a^y b^l$. Then there is a well-defined map $$\langle a,b,c \mid bab\inv = a^{n^l}, \; cac\inv = a\inv, \; cbc\inv = a^y b, \; c^2 = 1\rangle \to \Gamma$$ given by $$a \to a\qquad b \to b \qquad c \to \gamma.$$ As before, this map restricts to an isomorphism on the index $2$ (normal) subgroups $\langle a,b \rangle \subset H$ and $\langle a,b^l\rangle \subset \Gamma$, and sends $c \not \in \langle a,b \rangle$ to $\gamma \not \in \langle a,b^l \rangle$. It is therefore an isomorphism, so $$\Gamma \cong \langle a,b,c \mid bab\inv = a^{n^l}, \; cac\inv = a\inv, \; cbc\inv = a^y b, \; c^2 = 1\rangle,$$ as required.

\end{proof}

\section{Automorphisms of $G_n$}
\label{sec:AutGn}

In this section we compute the (topological) automorphism group of $G_n$. The connected component of the identity in $G_n$ is the subgroup of pure translations $((x,y),v) \to ((x+r,y),v)$ and is isomorphic to $\R$. Any automorphism $f:G_n \to G_n$ restricts to an element of $\Aut(\R) \cong \R^\ast$, and descends to an automorphism of the quotient\footnote{We are careful to note that $T = T_{1,n}$ depends on $n$.} $G_n/\R \cong \Aut(T)$, denoted $\tilde{f}\in \Aut(\Aut(T))$. These two pieces will end up characterizing $f$. We will first show that $\tilde{f}$ must be inner, and then prove that $\Aut(G_n) \cong \R^\ast \times \Aut(T)$. In the third subsection, we deduce the consequences for lattices $\Gamma \subset G_n$.

\subsection*{The automorphism group of $\Aut(T)$.} We first show that $\Aut(T)$ has no outer automorphisms.

\begin{lemma}
\label{lem:AutAutT}
    The (topological) outer automorphism group $\Out(\Aut(T))$ is trivial.
\end{lemma}

\begin{proof}
Let $f: \Aut(T) \to \Aut(T)$ be a topological automorphism. Set $H = \Stab(v_0)$. Then $f(H)$ is compact, and by an averaging argument, must fix some point, which we may assume is a vertex since $\Aut(T)$ acts without edge inversions. By conjugating $f$ (or rather, composing with an inner automorphism), we can assume $f(H)$ fixes $v_0$, that is, $f(H) \subset H$. Both $H$ and $f(H)$ are open and compact, hence have finite positive Haar measure. Then $r = [H: f(H)] = \mu(H)/\mu(f(H)) < \infty$. Let $g_i \in f\inv(H)$ be coset representatives so that $H = \cup_{i = 1}^r f(g_i)f(H)$ and $f\inv(H) = \cup_{i = 1}^r g_iH$. Since the $f(g_i)$ all lie in a point stabilizer, the $g_i$ have height change $0$. Let $w$ be the maximal element of $\cap_i^r \down(g_i\cdot v_0)$ (where $\down(v)$ denotes the vertices with height less than or equal to $h(v)$ on a coherently oriented line passing through $v$). We now show that $f\inv(H) = \Stab(w)$. 

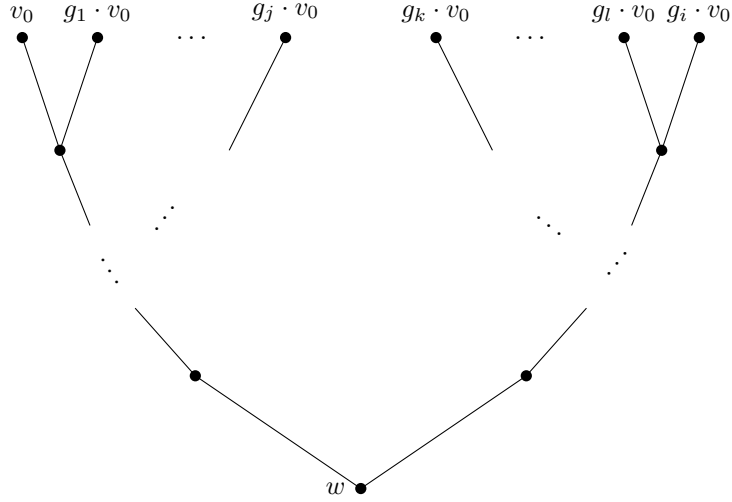
\begin{figure}
    \centering
   
\begin{tikzpicture}

\node[label=left:{$w$}] at (0,0)[circle,fill,inner sep=1.5pt]{};

\node[label=left:{}] at (-2.2,1.5)[circle,fill,inner sep=1.5pt]{};

\node[rotate=-20] at (-3.3,3){$\ddots$};

\node[label=left:{}] at (-4,4.5)[circle,fill,inner sep=1.5pt]{};

\node[label=above:{$v_0$}] at (-4.5,6)[circle,fill,inner sep=1.5pt]{};

\draw (0,0) -- (-2.2,1.5);
\draw (-2.2,1.5) -- (-3,2.4);
\draw (-4,4.5) -- (-3.6,3.5);
\draw (-4.5,6) -- (-4,4.5);

\node[label=left:{}] at (2.2,1.5)[circle,fill,inner sep=1.5pt]{};

\node[rotate=90] at (3.3,3){$\ddots$};

\node[label=left:{}] at (4,4.5)[circle,fill,inner sep=1.5pt]{};

\node[label=above:{$g_i \cdot v_0$}] at (4.5,6)[circle,fill,inner sep=1.5pt]{};

\draw (0,0) -- (2.2,1.5);
\draw (2.2,1.5) -- (3,2.4);
\draw (4,4.5) -- (3.6,3.5);
\draw (4.5,6) -- (4,4.5);

\node[label=above:{$g_1\cdot v_0$}] at (-3.5,6)[circle,fill,inner sep=1.5pt]{};

\node[label=above:{$g_j\cdot v_0$}] at (-1,6)[circle,fill,inner sep=1.5pt]{};

\node[label=above:{$g_{k}\cdot v_0$}] at (1,6)[circle,fill,inner sep=1.5pt]{};

\node[label=above:{$g_{l}\cdot v_0$}] at (3.5,6)[circle,fill,inner sep=1.5pt]{};

\draw (-3.5,6) -- (-4,4.5);
\draw (3.5,6) -- (4,4.5);

\draw (-1,6) -- (-1.75,4.5);
\draw (1,6) -- (1.75,4.5);

\node[rotate=-90] at (-2.5,3.6){$\ddots$};

\node[rotate=0] at (2.5,3.6){$\ddots$};

\node at (-2.25,6){$\ldots$};
\node at (2.25,6){$\ldots$};

\end{tikzpicture}

    \caption{\small{The tree with $v_0$, $w$, and $g_j\cdot v_0$ shown.}}
\end{figure}

Let $i$ be an index so that $w = \max(\down(v_0) \cap \down (g_i \cdot v_0))$ and let $l = d(v_0,w)$. We first show that $f\inv(H)$ acts transitively on $\up_l(w)$. Let $v \in \up_l(w)$. If the geodesic connecting $v_0$ to $v$ does not pass through $w$, then there is an element of $\Stab(g_i\cdot v_0) = g_i H g_i\inv \subset f\inv(H)$ taking $v_0$ to $v$. If the geodesic connecting $v_0$ to $v$ does pass through $w$, then there is an element of $\Stab(v_0) = H \subset f\inv(H)$ taking $v$ to $g_i\cdot v_0$. Since $g_i \in f\inv(H)$ takes $v_0$ to $g_i\cdot v_0$, this proves that $f\inv(H)$ acts transitively on $\up_l(w)$. Since $H = \Stab(v_0) \subset f\inv(H)$, the orbit-stabilizer theorem implies that for every $v \in \up_l(w)$, the subgroup $f\inv(H)$ contains all elements sending $v_0$ to $v$. Thus, $\Stab(w)\subset f\inv(H)$. The reverse inclusion follows directly from the coset partition $f\inv(H) = \cup_i^r g_iH$ and the definition of $w$.

Then $f\inv(H) = \Stab(w)$, and by conjugating again, we can assume $f(H) = H$. Since all vertex stabilizers are conjugate in $\Aut(T)$, we conclude that $f$ maps vertex stabilizers to (possibly different) vertex stabilizers. Furthermore, adjacency of vertices $v,w \in T$ is encoded algebraically as follows: $v$ lies directly above $w$ if and only if $\Stab(v) \subset \Stab(w)$ has index $n$. Then the map $g: T \to T$ specified by $h(v) = w$ when $f(\Stab(v)) = \Stab(w)$ is a graph automorphism. After conjugating $f$ by $g\inv$, we can assume $f$ fixes all point stabilizers setwise. Now let $g \in \Aut(T)$ be any automorphism, and let $v \in T$. Then $$\Stab(f(g)\cdot v) = f(g)\Stab(v) f(g)\inv = f(g)f(\Stab(v)) f(g\inv) = f(\Stab(g\cdot v)) = \Stab(g\cdot v)$$ so $f(g)\cdot v = g\cdot v$, for all $v\in T$ and $f \in \Aut(\Aut(T))$ is the identity. This completes the proof that $\Out(\Aut(T))$ is trivial.  

\end{proof}

\subsection*{The proof of Theorem \ref{thm:Automorphism}}

We will now show that $\Aut(G_n) \cong \R^\ast \times \Aut(T)$. Recall from Section \ref{sec:XnPrelim} that there is a semidirect decomposition $G_n = \R \rtimes \Aut(T)$, where the action of $\Aut(T)$ on $\R$ is given by $g \cdot x = n^{h(g)}x$. For any $r \in \R^\ast$, there is an automorphism of $G_n = \R \rtimes \Aut(T)$ given by $(x,g) \to (rx,g)$. Similarly, for a pure tree action $g \in \Aut(T) \subset G_n$, there is a corresponding inner automorphism $C_g$. These are the two factors of the decomposition $\Aut(G_n) \cong \R^\ast \times \Aut(T)$.

\begin{proof}[Proof of Theorem \ref{thm:Automorphism}]

Let $f \in \Aut(G_n)$. As discussed in the beginning of the section, we have automorphisms $f|_\R:\R \to \R$ and $\tilde{f}:\Aut(T) \to \Aut(T)$ given by restricting to and quotienting by the connected component of the identity. Let $f|_\R = r \in \R^\ast \cong \Aut(\R,+)$. Lemma \ref{lem:AutAutT} implies that there is some $g \in \Aut(T)$ so that $\tilde{f} = C_g$, where $C_g$ denotes the inner automorphism of $G_n$ given by conjugation by $g \in \Aut(T)\subset G_n$. 

Let $(x,k) \in G_n = \R \rtimes \Aut(T)$, for $x \in \R$ and $k \in \Aut(T)$. We now compute: $$C_g\inv \circ f(x,k) = C_g\inv (r\cdot x ,\tilde{f}(k)) = (n^{h(g\inv)} r \cdot x, g\inv \tilde{f}(k)g) = (n^{h(g\inv)} r \cdot x,k).$$ This shows that every $f \in \Aut(G_n)$ is the product of some $C_g\in \Aut(T) \subset \Aut(G_n)$ and some $n^{h(g\inv)} \cdot f|_\R \in \R^\ast \cong \Aut(\R,+)$. All that remains is to show any two of these elements commute: let $r \in \R^\ast$ and $C_g \in \Aut(T)$. Then $$C_g \circ r \circ C_g\inv (x,k) = (n^{h(g)}rn^{h(g\inv)}x, gg\inv k gg\inv) = (rx,k) = r(x,k),$$ so $\Aut(G_n) = \R^\ast \times \Aut(T)$, as required.

\end{proof}

Corollary \ref{cor:Corollary} now follows from Theorems \ref{thm:Classification} and \ref{thm:Automorphism}.

\begin{proof}[Proof of Corollary \ref{cor:Corollary}]
   Parts 1 and 2 are restatements of Theorems \ref{thm:LatticesAreUniformAndBS} and \ref{thm:Classification}. Part 3 follows from Theorem \ref{thm:Classification} and the observation that $f_r \circ \phi_{s,m} = \phi_{rs,m}$ for $f_r = (r,\text{id}_{\Aut(T)}) \in \R^\ast \times \Aut(T) \cong G_n$. 

   \textit{Part 4} states that if $m \neq m'$ both satisfy $(\ast)$, then there does not exist $\rho \in \Aut(G_n)$ so that $\rho \circ \phi_{1,m} = \phi_{1,m'}$. If such a $\rho = (r, g) \in \R^\ast \times \Aut(T)$ existed, it must have $r = 1$, and therefore be inner. By Theorem \ref{thm:Classification}, there is no inner automorphism $\rho$ with $\rho \circ \phi_{1,m} = \phi_{1,m'}$.

   \textit{Part 5.} Let $\Gamma_1,\Gamma_2\subset G_n$ be isomorphic lattices. Then $\Gamma_1 \cong \Gamma_2 \cong \BS(1,n^l)$ for some $l$. Let $\psi_1:\BS(1,n^l) \to G_n$ and $\psi_2:\BS(1,n^l) \to G_n$ be lattice embeddings with image $\Gamma_1$ and $\Gamma_2$ respectively. Then there are automorphisms $\rho_1,\rho_2 \in \Aut(G_n)$ and $m,m'$ satisfying $(\ast)$ so that $\rho_1 \circ \psi_1 = \phi_{1,m}$ and $\rho_2\circ \psi_2 = \phi_{1,m'}$. Since $\phi_{1,m} = \phi_{1,1} \circ \theta_m$ and $\theta_m$ is an automorphism, the image of $\phi_{1,m}$ is exactly the standard lattice $\langle a,b^l\rangle$. Then $\rho_1\circ \psi_1$ and $\rho_2 \circ \psi_2$ have the same image. 
\end{proof}

\subsection*{Automorphisms of $\BS(1,n)$.}

In \cite{CollinsAutBS}, Collins computed the automorphism group of $\BS(1,n)$. In this section, we will briefly discuss this result and its relationship to $\Aut(G_n)$. In particular, we will see that some automorphisms of $\BS(1,n)$ extend to automorphisms of $G_n$, but (as implied by our main results) some do not.

Write $n = p_1^{\epsilon_1}p_2^{\epsilon_2}\cdots p_{m}^{\epsilon_m}$ for $p_i$ distinct primes, and let $C,D,Q_i = \theta_{p_i}$ be the automorphisms $\BS(1,n)\to \BS(1,n)$ defined by 

$$\begin{array}{l}
C(a) = a \\
C(b) = ab
\end{array},\qquad \begin{array}{l}
D(a) = a\inv \\
D(b) = b
\end{array} \qquad  \text{and} \qquad \begin{array}{l}
Q_i(a) = a^{p_i} \\
Q_i(b) = b
\end{array}.$$

Collins showed that $D$ is an outer automorphism, $C$ is inner if and only if $n = 2$, and $Q_i$ is inner if and only if $n = p_i$. These account for all outer automorphisms.

\begin{theorem}{Collins, \cite[Proposition 4]{CollinsAutBS}}
    Let $A$ and $B$ be the inner automorphisms of $\BS(1,n)$ corresponding to $a$ and $b$ respectively. Then $\Aut(\BS(1,n)) = \langle A,B,C,D,Q_1,\ldots, Q_m\rangle$.
\end{theorem}

We will now analyze which automorphisms of $\BS(1,n)$ extend to automorphisms of $G_n$.

\begin{enumerate}
    \item As inner automorphisms of $\BS(1,n)$, the maps $A$ and $B$ extend to inner automorphisms of $G_n$.
    \item The outer automorphism $D$ extends to $(-1,\text{id}_{\Aut(T)}) \in \R^\ast \times \Aut(T) \cong \Aut(G_n)$.
    \item Since the outer automorphism $C$ fixes $a$, and $v_0 = \max(\Fix(C(a))\cap \Axis(C(b)))$, the conjugacy invariants $(s,m)$ of $\phi_{1,1} \circ C$ are $(1,1)$, so (by Theorem \ref{thm:Classification}) $\phi_{1,1} \circ C$ is conjugate to $\phi_{1,1}$, and $C$ extends to an inner automorphism of $G_n$. This is explained by the fact that (when $n > 2$) the hyperbolic elements $ba$ and $b$ are not conjugate to in $\BS(1,n)$, but are conjugate in $G_n$ (which contains all pure tree automorphisms). 
    \item If $n \neq p_i$, the embeddings $Q_i = \theta_{p_i}$ do not extend to an automorphism of $G_n$, by Corollary \ref{cor:Corollary}. If $n = p_i$, then $Q_i$ is the inner automorphism given by conjugation by $b$. 
\end{enumerate}

\section{$G_n$ is not linear.}
\label{sec:NotLinear}

In this section we will prove that $G_n$ is not linear. Precisely,

\begin{theorem}
\label{thm:NotLinear}
    Let $d \in \N$ and let $K$ be a field with characteristic $0$. Then there is no faithful linear representation $G_n \to \GL_d(K)$.
\end{theorem}

The proof of based on Jordan's theorem \cite{Jordan}.

\begin{theorem}[Jordan]
    Let $K$ be a field of characteristic $0$. For all $d \in \N$, there is some $q$ (depending only on $d$) so that every finite subgroup $H \subset \GL_d(K)$ has an abelian subgroup $A \subset H$ with index at most $q$. 
\end{theorem}

Jordan's theorem is usually stated for $K = \C$, but the general case follows directly: if $G \subset \GL_d(K)$ is finite, consider the subfield $K' \subset K$ generated over $\Q$ by the entries of elements of $G$. Then $K'$ has finite transcendence degree over $\Q$, and therefore embeds into $\C$. Then $G \subset \GL_n(K') \subset \GL_n(\C)$.

We will exhibit finite subgroups of $\Aut(T)\subset G_n$ which do not have large abelian subgroups. For $k \in \N$, set $\up_{\leq k}(v_0) = \{v \in \up(v_0) \mid h(v)\leq k\}$, and $H_k = \Aut(\up_{\leq k}(v_0))$, a finite group. Choose an identification $f_w:\up(v_k) \to \up(w)$ for $w \in \up_k(v_0)$. We can exhibit $H_k$ as a subgroup of $\Aut(T)$ by extending $g \in H_k$ by $g(v) = v$ when $v \not \in \up(v_0)$, and by $g(v) = f_{g(w)}\circ f_w\inv(v)$ when $v \in \up(w)$ for some $w \in \up_k(v_0)$. The goal of the remainder of this section is to show that for large enough $k$, the subgroups $H_k$ violate the conclusion of Jordan's theorem. 

\begin{lemma}
\label{lem:countHk}
    Let $H_k = \Aut(\up_{\leq k}(v_0))$. If $n \geq 1$, then $|H_k| = (n!)^{k}\cdot n^{k-1}$.
\end{lemma}

\begin{proof}
    Clearly $|H_0| = 1$. Consider the action of $H_k$ on the set of $n$-tuples of distinct elements of $\up_1(v_0)$. This action is transitive, and the stabilizer of any $n$-tuple is isomorphic to $\prod_{i = 1}^n H_{k-1}$. Then Orbit-Stabilizer gives the recurrence relation $|H_k| = n! \cdot n\cdot |H_{k-1}|$, from which we conclude $|H_k| = (n!)^{k}\cdot n^{k-1}$.
\end{proof}

Let $a \in H_k$ denote the restriction of the standard action of $a \in \BS(1,n)$ to $\up_{\leq k}(v_0)$ (by slight abuse of notation, we do not distinguish notationally between the $a \in H_k$ for various $k$). We will prove that $H_k$ has no large abelian subgroups $A$ by showing that such a subgroup must contain some $a^m$, and that the centralizer of $a^m$ (which contains $A$) is small. We will now bound the order of the centralizer of $a^m$.

\begin{lemma}
\label{lem:boundCentralizer}
    Let $m \in \N$ and let $l,s,t$ so that $m\cdot s = n^l\cdot t$ and $\gcd(n,t) = 1$. If $l < k$, then $|C_{H_k}(a^m)|\leq n^k \cdot |H_l| = n^{k+l-1} \cdot (n!)^l$.
\end{lemma}

\begin{proof}

Since $C_{H_k}(a^m) \subset C_{H_k}((a^{m})^s)$, it suffices to bound the order of $C_{H_k}(a^{tn^l})$. We will decompose $C_{H_k}(a^{tn^l})$ according to what its elements do on $\up_l(v_0)$ and what they do above it, see Figure \ref{fig:boundCentralizer}. Let $w \in \up_l(v_0)$. Using the explicit form of $a$ given in Section \ref{subsec:ActionOfBS}, we see that $a^{n^l}$ fixes $\up_l(v_0)$ pointwise, and acts transitively forever on $\up(w)$. Since $\gcd(n,t) = 1$, $(a^{n^l})^t$ also acts transitively forever on $\up(w)$.

\begin{figure}
\label{fig:boundCentralizer}
    \centering
    \begin{tikzpicture}

\begin{scope}
  [level distance=10mm,level/.style={sibling distance=15mm/#1},grow'=up]
  \coordinate
    child foreach \x in {0,1}
      {child foreach \y in {0,1}
        {child foreach \z in {0,1}}};
\end{scope}

\begin{scope}
  [level distance=10mm,level/.style={sibling distance=15mm/#1},grow'=up,xshift=3cm]
  \coordinate
    child foreach \x in {0,1}
      {child foreach \y in {0,1}
        {child foreach \z in {0,1}}};
\end{scope}

\node[label=below:{$v_0$}] at (1.5,-1.2)[circle,fill,inner sep=1.5pt]{};

\node[rotate=0] at (2.5,-.5){$\iddots$};

\node[rotate=0] at (.5,-.5){$\ddots$};


\draw[dotted] (-2,10mm) -- (5,10mm);

\node[label=below:{height $l$}] at (-2,10mm){};

\node[label=below:{$w$}] at (.75,10mm)[circle,fill,inner sep=1.5pt]{};

\draw[dashed] (0,.6) rectangle (1.5,3.2);

\end{tikzpicture}
    \caption{\small{The tree $T$ with height $l$ marked, and $\up(w)$ marked. $a^{n^l}$ acts trivially on $\up_{\leq l}(v_0)$, and acts transitively forever on $\up(w)$. Then an element commuting with $a^{n^l}$ can do anything to $\up_{\leq l}(v_0)$. If such an element fixes $w$, it is constrained by Lemma \ref{lem:CentralizerOfa}.}}
    \label{fig:enter-label}
\end{figure}
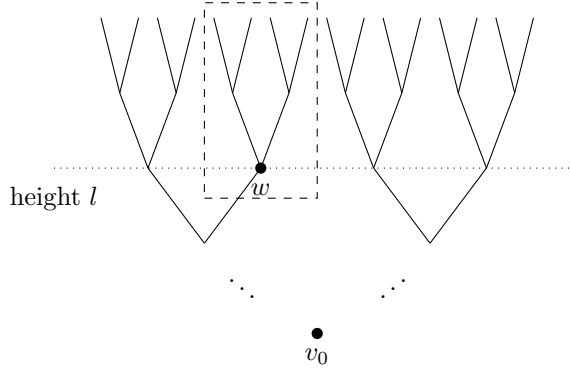

The arguments from Section \ref{sec:ClassificationOfEmbeddings} (step 3) show that the restriction $a^{tn^l}$ to $\up(w)$ (or to any element acting transitively forever on $\up(w)$) is conjugate to the action of $a$ on $\up(v_0)$. In particular, we may apply Lemma \ref{lem:CentralizerOfa} (in its truncated form) to conclude that the set of elements in $\Aut(\up_{\leq k-l}(w))$ commuting with $a^{tn^l}|_{\up(w)}$ is isomorphic to $\Z / n^{k-l}$. Extending these elements to $\up_{\leq k}(v_0)$ by the identity map on $\up_{\leq k}(v_0) \setminus \up(w)$, we obtain a subgroup $\Z / n^{k-l} \subset H_k$ for each $w \in \up_l(v_0)$. These subgroups have disjoint support, so there is a subgroup $N = \prod_{w \in \up_l(v_0)}\Z/n^{k-l} \subset C_{H_k}(a^{tn^l})$. Since $N$ is exactly the subgroup of $C_{H_k}(a^{tn^l})$ fixing $\up_l(v_0)$, it is normal. 

For $w \in \up_l(v_0)$, let $f_w:\up_{\leq k-l}(v_l) \to \up_{\leq k-l}(w)$ be a graph isomorphism conjugating the two actions of $a^{tn^l}$. Any $g \in H_l$ can be extended to $\up_{\leq k}(v_0)$ by $g(v) = f_{g(w)}\circ f_w\inv (v)$. This gives an inclusion $i:H_l \to H_k$. Since the $f_w$ were chosen to conjugate the actions of $a^{tn^l}$, this extension lies in $C_{H_k}(a^{tn^l})$.

Let $g \in C_{H_k}(a^{tn^l})$. Let $g' \in H_l$ denote the restriction of $g$ to $\up_{\leq l}(v_0)$. Then $i(g')\inv g $ commutes with $a^{tn^l}$ and fixes $\up_l(v_0)$. Then $i(g')\inv g \in N$. We then have a split short exact sequence $$1 \to \prod_{w \in \up_l(v_0)}\Z/n^{k-l} \to C_{H_k}(a^{tn^l}) \to H_l \to 1.$$ Hence, $|C_{H_k}(a^{tn^l})|\leq n^l \cdot n^{k-l}\cdot |H_l|$

\end{proof}

We are now prepared to prove Theorem \ref{thm:NotLinear}.

\begin{proof}[Theorem \ref{thm:NotLinear}]
    Say there were a faithful linear representation $\Aut(T)\to \GL_n(K)$. Let $q$ be as in Jordan's theorem, and let $L \in \N$ be large enough so that for every $m = 1, 2,\ldots, q$, there are $t,s$ and $l \leq L$ so that $m\cdot s = t\cdot n^l$. For $k \gg L$, Jordan's theorem gives an abelian group $A \subset H_k$ of index at most $q$. Then $a^m \in A$ for some $m = 1,\ldots q$. Lemmas \ref{lem:countHk} and \ref{lem:boundCentralizer} now give $$[H_k:A] \geq [H_k: C_{H_k}(a^m)] \geq \frac{(n!)^k\cdot n^{k-1}}{n^k\cdot (n!)^L\cdot n^{L-1}}\geq \frac{(n!)^{k-L}}{n^L} > q,$$ contradicting Jordan's theorem.
    
\end{proof}

\bibliographystyle{alpha}
\bibliography{bib}

\end{document}